\documentclass[11pt,a4paper]{article}
\usepackage[english]{babel}
\usepackage{bbold}
\usepackage{enumitem}

\usepackage{tikz,tkz-tab}
\usepackage{amsmath}
\usepackage{amsthm}
\usepackage{comment}
\usepackage{amssymb}
\usepackage{amsfonts}
\usepackage{ upgreek }
\input{regstruct_new.sty}
\usepackage{pgf,tikz}
\usetikzlibrary{matrix,automata,arrows,positioning,calc,patterns,fit,calc,
decorations.pathreplacing,decorations.pathmorphing,decorations.markings,matrix,
tikzmark,shadows.blur}
\usetikzlibrary{cd}
\usepgflibrary{fpu}
\usepackage{verbatim}
\usepackage{mathtools}

\usepackage{bigints}
\usepackage{stmaryrd}

\newcommand\R{\mathbb{R} }
\newcommand\N{\mathbb{N} }

\newcommand\Z{\mathbb{Z} }

\usepackage{scalerel}
\newcommand{\pe}{\mathbin{\scaleobj{0.7}{\tikz \draw (0,0) node[shape=circle,draw,inner sep=0pt,minimum size=8.5pt] {\tiny $=$};}}}
\newcommand{\pne}{\mathbin{\scaleobj{0.7}{\tikz \draw (0,0) node[shape=circle,draw,inner sep=0pt,minimum size=8.5pt] {\tiny $\neq$};}}}
\newcommand{\pl}{\mathbin{\scaleobj{0.7}{\tikz \draw (0,0) node[shape=circle,draw,inner sep=0pt,minimum size=8.5pt] {\tiny $<$};}}}
\newcommand{\pg}{\mathbin{\scaleobj{0.7}{\tikz \draw (0,0) node[shape=circle,draw,inner sep=0pt,minimum size=8.5pt] {\tiny $>$};}}}

\newcommand{\pge}{\mathbin{\scaleobj{0.7}{\tikz \draw (0,0) node[shape=circle,draw,inner sep=0pt,minimum size=8.5pt] {\tiny $\geqslant$};}}}

\newtheorem{theorem}{Theorem}[section]
\newcommand\numberthis{\addtocounter{equation}{1}\tag{\theequation}}
\newtheorem{lemma}[theorem]{Lemma}
\newcommand{\vide}[1]{}

\newtheorem{proposition}[theorem]{Proposition}

\newtheorem{corollary}[theorem]{Corollary}
\newtheorem{remark}[theorem]{Remark}

\newtheorem*{assumption}{Assumption}

\usepackage{graphicx}
\title{Concentration around a stable equilibrium for the non-autonomous $\Phi_3^4$ model}
\author{Dimitri Faure\footnote{Institut Denis Poisson, University of Orléans, University of Tours and CNRS, France. \\ \text{} \hspace{0.12em}   \ \ \   Email: \texttt{dimitri.faure@univ-orleans.fr} }}
\begin{document} 
\maketitle
\begin{abstract}
    We consider time-dependent singular stochastic partial differential equations on the three-dimensional torus. These equations are only well-posed after one adds renormalization terms. In order to construct a well-defined notion of solution, one should put the equation in a more general setting. In this article, we consider the paradigm of paracontrolled distributions, and get concentration results around a stable deterministic equilibrium for solutions of non-autonomous generalizations of the $(\Phi_3^4)$ model. \textcolor{black}{Specifically, we obtain Gaussian-type tail bounds.}
\end{abstract}
\tableofcontents
\section{Introduction}
Since Hairer's landmark aticles on regularity structures \cite{H1,H2}, numerous \textcolor{black}{papers} have been written to solve existence and uniqueness problems for important singular stochastic partial differential equations (SPDEs). Singular SPDEs are characterized by the presence of undefined products as soon as one tries to solve them using \textcolor{black}{standard} methods. Indeed, the space-time white noise $\xi$ which appears in many SPDEs has the regularity of a distribution, and we expect solutions to a lot of stochastic PDEs to also have the regularity of a distribution: all non-linerarities in the equation are hence undefined, and consequently the equation itself. A classical example of singular SPDE is the $(\Phi_3^4)$ model which is written
$$\partial_t \varphi=\Delta \varphi-\varphi^3+\varphi+ \sigma \xi$$
where \textcolor{black}{$\sigma>0$ is a parameter and $\varphi$ is expected to be only a distribution on $[0,T]\times \mathbb{T}^3$ due to the noise term. The ill-defined term is therefore $\varphi^3$, since one cannot consider the cube of a distribution.} \textcolor{black}{To solve this conceptual difficulty, one has to add to the equation new diverging terms called "renormalization terms", and prove that such a modified equation admits a solution.} \\

While Hairer's method is extremely powerful for existence and uniqueness results, it is an abstract method which is extremely general, making it less useful when one is studying quantitative properties of the solutions. In parallel with Hairer's work, Kupiainen developed a method to study SPDEs based on Wilsonian renormalization group analysis \cite{K}, while Gubinelli, Imkeller and Perkowski put forward the theory of paracontrolled distributions in \cite{GIP}, which was used in \cite{CC} to get local-in-time well-posedness results for a specific SPDE. The three methods presented have the same notion of solutions, but for the problem we study, the theory of paracontrolled distributions is the handiest one. \textcolor{black}{The main idea of paracontrolled distribution is to decompose the standard product of two functions into three separate subproducts that are easier to deal with, especially when one of the functions at hand is a distribution. For} the $(\Phi_3^4)$ model, \textcolor{black}{this method allows for} an almost explicit expression of the solutions, and since we are interested in concentration around an equilibrium results, we prefer this to abstract writings of solutions.\\

Our work is at the crossroads of two series of papers. In \cite{BN1,BN2}, the authors study concentration results around an equilibrium for non-autonomous generalizations of the $(\Phi_d^4)$ problem in dimension $d=1$ and $d=2$, and describe some bifurcation phenomena between equilibrium branches. In these dimensions, the equation is not singular ($d=1$) or the solution is a distribution but it is "almost" a continuous function ($d=2$), and therefore we can guess what renormalization terms look like without introducing new theories. In \cite{MW,MWX}, the authors study the $(\Phi_3^4)$ model, with a particularly useful diagram-based formalism. More precisely, in \cite{MW}, Mourrat and Weber get an a priori bound for the $(\Phi_3^4)$ model on the torus which is independent of the initial condition while rewriting the paracontrolled calculus associated with solving $(\Phi_3^4)$ in the new formalism. In \cite{MWX}, \textcolor{black}{Mourrat, Weber and Xu} compute explicit bounds on the moments of the Hölder norm of the standard stochastic integrals that are needed to study $(\Phi_3^4)$. \\

In this article, we are interested in getting concentration results around a deterministic \textcolor{black}{stable} equilibrium for the solution of a non-autonomous generalization of $(\Phi_d^4)$ in dimension $d=3$. \textcolor{black}{Let us denote $\varphi$ a solution to a non-autonomous generalized $(\Phi_3^4)$ equation and $\overline{\phi}$ an attractive solution to the associated deterministic PDE with the same constant initial condition as $\varphi$. Our} main result is that for $\sigma \ll 1$, \textcolor{black}{we have with a high probability that $\varphi$ will stay at most at a distance $\sigma$ of $\overline{\phi}$ during $[0,T]$ for an appropriate Hölder norm}.\\

We emphasize that the existence of an attractive deterministic solution is not guaranteed, and that the regime $\sigma\approx 1$ also deserves to be considered, but we highlight that concentration results around stable equilibrium are the first step towards the study of scenarii involving bifurcations. \textcolor{black}{To our knowledge, this problem has not been studied yet, the theory of singular SPDEs being a still young and mainly qualitative theory, especially when one considers non-autonomous equations. Finally, we underscore that the concentration results we eventually get are stronger than the strong deviations ones found for instance in \cite{H2}, since they hold for Hölder norms of strictly positive exponents.}
\\

The work is divided as follows for the main parts of the article: in the second section, we introduce the main notations, and state the main concentration theorems, while in the third section we prove them. With regards to the Appendix, it recalls definitions and important results in the theory of Besov spaces, paracontrolled calculus and Wiener chaos.

\section{Setting and main results \label{resultats}}
\subsection{Setting and main theorems \label{dth}}
We are considering the following general equation in dimension $3$:
\begin{equation}\left\{\begin{array}{ll}&\partial_t \varphi = \Delta \varphi+R(t, \varphi)+\sigma\xi \label{mne} \\
&\varphi(0)=\varphi_0\end{array} \right. \end{equation}
where $\varphi_0$ is constant in space and $R$ is a third degree polynomial with non-constant coefficients $$R(t,\varphi)=\textcolor{black}{-}\varphi^3+\textcolor{black}{b}_2(t) \varphi^2+\textcolor{black}{b}_1(t)\varphi+\textcolor{black}{b}_0(t).$$
Here we \textcolor{black}{will consider the equation on a finite time interval $[0,T]$ and} will assume that $\textcolor{black}{b}_0,\textcolor{black}{b}_1,\textcolor{black}{b}_2:[0,T]\to \R$ are differentiable (and therefore continuous). Since we are interested in concentration results around an equilibrium for potentially \textcolor{black}{small or potentially} large times $T$, we will make the dependency in $T$ of all constants explicit.

\begin{remark}\label{mg}
    The degree of $R$ has to be odd for the solution of the associated PDE not to explode, and the equation is no longer subcritical if the degree is greater or equal to five. Since there is no known theory to solve non-subcritical SPDE, our setting is \textcolor{black}{essentially} the most general we can expect assuming the initial condition is constant. \textcolor{black}{Nonetheless, we can improve a bit our results by replacing the $-1$ in front of $\varphi^3$ by any time-dependent, negative, differentiable coefficient (see Remark \ref{ga3}).}
\end{remark}
A key remark of singular SPDE theory, is that there is no solution to our equation in the classical sense. Indeed, knowing that the space regularity of $\xi$ is $\mathcal{C}^{-\frac{5}{2}-\varepsilon}$ \textcolor{black}{(see Appendix \ref{PL} for a definition of Hölder spaces of negative exponents)}, Schauder's theory gives us that the expected space regularity of the solution to \eqref{mne} is $\mathcal{C}^{-\frac{1}{2}-\varepsilon}$ making all the non-linear terms undefined. This underscores the limits of the standard formalism when describing phenomena associated with stochastic PDEs, and \textcolor{black}{to solve this difficulty, we have to introduce the renormalization procedure. We however prefer to simplify our equation before detailing it, since the computations needed to get our results are sometimes long.} \\

We \textcolor{black}{can easily} get rid of the constant term and take an initial condition equal to $0$. Indeed, for $\phi=\varphi-\overline{\phi}$ where $\overline{\phi}$ is a (deterministic) constant in space solution of
\begin{equation}\left\{ \begin{array}{ll}&\partial_t \overline{\phi}=\Delta \overline{\phi}+R(t,\overline{\phi}) \\
&\overline{\phi}(0)=\varphi_0\end{array} \right. \label{eqdet}\end{equation}
we have that
$$\partial_t \phi=\Delta \phi+(R(t,\overline{\phi}+\phi)-R(t,\overline{\phi}))+\sigma\xi,$$
\textcolor{black}{and since} $\phi\in \R\mapsto R(t,\overline{\phi}+\phi)-R(t,\overline{\phi})$ is a polynomial in $\phi$ equal to $0$ in $\phi=0$, we get that
\begin{equation}\label{eqphi}\left\{\begin{array}{ll}&\partial_t \phi=\Delta \phi+\tilde{R}(t,\phi)+\sigma\xi  \\ & \phi(0)=\overline{0} \end{array}\right.\end{equation}
with $\tilde{R}(t,\phi)=-\phi^3+f_2(t)\phi^2+f_1(t)\phi$. We now take $\sigma>0$ and write $a:=f_1$. \textcolor{black}{In the next paragraph, we state an hypothesis on $\overline{\phi}$ that we will assume to hold for the rest of this article. While we could get some concentration results without this hypothesis, they would be far less precise than the ones we get with it (see Remark \ref{hypa}).}

\begin{assumption}
\textcolor{black}{The function $\overline{\phi}$ is a stable solution of \eqref{eqdet}, that is to say that the linearization of \eqref{eqdet} around $\overline{\phi}$ could be rewritten $\partial_t\psi=O(t)\psi$ where the operator $O(t)$ is linear with only negative eigenvalues for all $t\in[0,T]$. In our setting, $\overline{\phi}$ is constant in space and is therefore stable if and only if $a(t)<0$ for all $t\in[0,T]$.}
\end{assumption}

\textcolor{black}{Now that we simplified our equation, we will introduce a key ingredient of renormalization, which is a collection of specific space-time distributions called symbols.} \textcolor{black}{Those symbols are written as limits of smooth functions whose definitions require a regularized white noise $\xi_n$. Formally, for $\rho:\R\times\mathbb{T}^3\to \R$ a compactly supported and smooth function of integral $1$, we denote $\rho_n(t,x)=n^4\rho(n^2t,nx)$ and we take $\xi_n(t,x)=(\rho_n*\xi)(t,x)$. We then have that $\xi_n$ is smooth and that $\xi_n\rightharpoonup \xi$}. All the symbols we will need in our computations can be built from $\RSI_n$ solution of
$$\left\{\begin{array}{ll}& \partial \RSI_n=\Delta \RSI_n +a(t)\RSI_n+\sigma\xi_n \\ & \RSI_n(0)=\overline{0}\end{array}\right.$$
Here $\RSI_n$ is a smooth space-time function, but taking the limit $n\to \infty$ we get \textcolor{black}{that} $\RSI=\lim \RSI_n$ \textcolor{black}{is} a space-time distribution. \textcolor{black}{While the definition we here gave of $\RSI_n$ is implicit, the non-autonomous heat equation does admit an explicit expression. We can therefore write $\RSI_n=I(\xi_n)$ where we have for $(e^{t\Delta})_{t\ge 0}$ the heat semigroup and $\alpha(t,u)=\int_u^t a(s)\mathrm{d}s$} $$\textcolor{black}{I(f)(t)=\int_0^t e^{\alpha(t,u)}e^{(t-u)\Delta}f(u)\mathrm{d}u,}$$ \textcolor{black}{the solution to $(\partial_t-\Delta-a(t))g=f(t)$ with initial condition $g(0)=0$.}\\

We will now construct the other symbols, but in order to do that we \textcolor{black}{first} need to introduce elements of paracontrolled calculus. The main tool of paracontrolled calculus is the decomposition of the standard product into three bilinear operators. More precisely, for two functions $f$ \textcolor{black}{and} $g$ satisfying certain regularity hypotheses we have
$$fg= f\pl g+ f\pe g+f\pg g$$
with $\pl$ and $\pg$ the paraproducts and $\pe$ the resonant product. \textcolor{black}{These three operators} are expressed using the decomposition of functions in Hölder spaces into infinite sums of Paley-Littlewood blocks. \textcolor{black}{While} the technical details are left for the Appendix, we want to underscore that $\pe, \pl$ and $\pg$ are relevant for our study because they interact particularly well with one another, so we can use technical results such as Proposition \ref{com} to give meaning to products that are supposed to be undefined. \\

Now that we have a definition for $\RSI$ and that we introduced the resonant product and the operator $I$, we can build all the symbols we need. They all belong to $\mathcal{T}=\{\RSI,\RSV,\RSY,\RSIW,\RSWV,\RSVW,\RSWW\}$ where we have following \cite{MWX}:
\begin{itemize}
    \item[] $\RSI(t):=\lim_{n\to+\infty} \RSI_n$
    \item[] $\RSV(t):=\lim_{n\to\infty} \RSV_n(t) =\lim_{n\to+\infty}\RSI_n(t)^2-c_n(t)$
    \item[] $\RSY(t):=\lim_{n\to\infty}I(\RSI_n^2-c_n)(t)$
    \item[] $\RSIW(t):=\lim_{n\to+\infty}\RSIW_n=\lim_{n\to+\infty} I(\RSI_n^3-3c_n \RSI)(t)$ 
    \item[] $\RSVW(t):=\lim_{n\to+\infty} \RSIW_n(t)\pe \RSI_n(t)$
    \item[] $\RSWV(t):=\lim_{n\to+\infty} \RSWV_n(t)=\lim_{n\to\infty} I(\RSV_n)(t)\pe \RSV_n(t)-2\tilde{c}_n(t)$
    \item[] $\RSWW(t):=\lim_{n\to\infty} \RSIW_n(t) \pe \RSV_n(t)-6 \tilde{c}_n(t) \RSI_n$
\end{itemize}
with $c_n(t):=\mathbb{E}(\RSI_n(t)^2)$ \textcolor{black}{proportional to} $\sigma^2$ and $\tilde{c}_n(t):=\mathbb{E}(I(\RSV_n)(t)\pe \RSV_n(t))$ \textcolor{black}{proportional to} $\sigma^4$, \textcolor{black}{all other parameters fixed (see \eqref{hom})}. \textcolor{black}{Giving a more explicit expression of $c_n(t)$ and $\tilde{c}_n(t)$ would be tedious, but we can show by adapting \cite{MWX}[p22 \& p27] where it is proven for the autonomous equation, that for $t>0$ we have $c_n(t)\asymp n$ and $\tilde{c}_n(t)\asymp \ln(n)$, all other parameters fixed.}\\

\textcolor{black}{As we see, the definition of most symbols of $\mathcal{T}$ involves diverging in $n$ terms. It is in fact a key example of renormalization procedure: the symbol $\RSI$ being a distribution, we cannot consider its square as a possible definition for $\RSV$, but we observe that $\RSI_n^2$ converges up to a deterministic, diverging in $n$ constant. As we will see later, the renormalization of \eqref{eqphi} will essentially amount to an indirect renormalization of symbols.}
\begin{remark}
    We will sometimes need in our computations $\RSW_n=\RSI_n^3-3c_n$ \textcolor{black}{and its limit $\RSW$}. \textcolor{black}{Even though $\RSW$ is a well defined space-time distribution, we don't put it in} $\mathcal{T}$ because it is very irregular in time and we cannot therefore give meaning to "$\RSW(t)$" for $t\in[0,T]$.
\end{remark}
Finally, we introduce for all these objects their chaos decomposition (CD). The precise definitions are left for the Appendix, but in few words we can write all our symbols as finite sums of space-time distributions called Wiener chaos, that are homogeneous in $\sigma$ \textcolor{black}{and whose moments can be bounded easily}. \textcolor{black}{The list of $k$ such} that $\tau\in \mathcal{T}$ admits a component in the Wiener chaos of order $k$ \textcolor{black}{is called the chaos decomposition of $\tau$, and the largest element of the list} is denoted $n_{\tau}$ and always corresponds to the number of leaves of $\tau$. \textcolor{black}{We furthermore denote $\mathcal{H}_k$ the Wiener chaos of order $k$ and $\Pi_k$ the projection on $\mathcal{H}_k$.} \\

\textcolor{black}{For a process $(\tau(t))_{t\ge 0}$ we denote $|\tau|$ the greatest $\alpha\in \R$ such that $\tau(t)\in \mathcal{C}^{\beta}$ for all $\beta<\alpha$ and $t\ge0$. Now that all the notations are introduced,} we \textcolor{black}{consider} the following table, with all information inside found in \cite{MWX}: 

\begin{equation*}
\begin{array}{|c||c|c|c|c|c|c|c|}
\hline
     \tau(t) & \RSI(t) & \RSV(t) &  \RSY(t) & \RSIW(t)& \RSVW(t) & \RSWV(t) & \RSWW(t) \\ \hline 
      | \tau | & -1/2 & -1 & 1 & 1/2 & 0 & 0 & -1/2 \\ \hline      \text{CD} & \mathcal{H}_1 & \mathcal{H}_2 & \mathcal{H}_2& \mathcal{H}_3& \mathcal{H}_2 \bigoplus \mathcal{H}_4 & \mathcal{H}_2\bigoplus \mathcal{H}_4 &  \mathcal{H}_1 \bigoplus \mathcal{H}_3 \bigoplus \mathcal{H}_5 \\ \hline
      n_{\tau} & 1 & 2 & 2 & 3 & 4 & 4 & 5 \\ \hline
\end{array}
\end{equation*}

\begin{remark}
    If we consider a space-time setting as in \cite{H2} \textcolor{black}{and we denote $\mathcal{C}_{\mathfrak{s}}^{\alpha}$ the space of functions on $\R\times \mathbb{T}^3$ which are $\alpha$-Hölder in space and $\frac{\alpha}{2}$-Hölder in time,} we can only expect $\RSI$ to be in $\mathcal{C}_{\mathfrak{s}}^{-\frac{1}{2}-\varepsilon}(\mathbb{T}^3\times \R)$. \textcolor{black}{Therefore,} its \textcolor{black}{Hölder regularity in time $t\in \mathbb{R}$} is \textcolor{black}{only} $-\frac{1}{4}-\frac{\varepsilon}{2}$ \textcolor{black}{for all $\varepsilon>0$}, and the notation "$\RSI(t)$" is a priori meaningless. Thanks to the \cite{MWX}[\textcolor{black}{(68)}] \textcolor{black}{combined with \cite{MWX}[(42)]}\textcolor{black}{,} we know that "$\RSI(t)$" \textcolor{black}{exists and is} an element of $\mathcal{C}^{-\frac{1}{2}-\varepsilon}$.
\end{remark}
Now that we have all the definitions needed, we will state our first concentration result on the spatial norm of the symbols of $\mathcal{T}$. \textcolor{black}{If this seems a priori unrelated to our main result, the symbols of $\mathcal{T}$ are ubiquitous in the equations characterizing $\phi$, and therefore a precise control on them imply a precise control on $\phi$.}
\begin{theorem}\label{cs} \textcolor{black}{Let us write $a_{+}=-\sup a$ and $a_{-}=-\inf a$.} 
    For all $\tau\in \mathcal{T}$, $\lambda\in (0,1)$ and $\alpha<|\tau|-\lambda$, there exists $c_{\tau}:=c_{\tau}(\lambda,\textcolor{black}{a_+},\textcolor{black}{a_-})$ \textcolor{black}{and} $d_{\tau}:=d_{\tau}(\lambda,\textcolor{black}{a_+},\textcolor{black}{a_-})>0$ independent of $T$ such that for $k\le n_{\tau}$ and $h>0$
    $$\mathbb{P}\left(\sup_{t\in [0,T]} \|\Pi_k\tau(t)\|_{\mathcal{C}^{\alpha}}>h^{k}\right)\le  d_{\tau}\exp\left(-\frac{c_{\tau}}{T^{\frac{\lambda}{k}}}\frac{h^2}{\sigma^2}\right).$$
\end{theorem}
\begin{remark}
    \textcolor{black}{For a given $\tau$, we can use Theorem \ref{cs} for all $k\in \llbracket 1, n_{\tau}\rrbracket$ to get $\mathbb{P}\left(\sup_{t\in [0,T]} \|\tau(t)\|_{\mathcal{C}^{\alpha}}>n_{\tau}(h+h^{n_{\tau}})\right)\le  n_{\tau}d_{\tau}\exp\left(-\frac{c_{\tau}}{\max(T^{\lambda},T^{\frac{\lambda}{n_{\tau}}})}\frac{h^2}{\sigma^2}\right)$ (see the proof of Corollary \ref{corp} for details).}
\end{remark}

Now that we have introduced the symbols of $\mathcal{T}$ and established that we have a strong control on their norms, we can go back to our main equation. \textcolor{black}{As we said, $\varphi$ and therefore $\phi$ are expected to be only distributions \textcolor{black}{in Schwarz sense}, making the non-linearities in their equation undefined. To deal with them, we have to write $\phi$ as the sum of a regular part, solution to an explicit equation, and an irregular part, made of symbols. Since we know how to deal with powers of symbols using renormalization, we can eventually give a meaning to all previously undefined terms, and therefore have a notion of solution for \eqref{eqphi}.}  \\

\textcolor{black}{Let us first consider $\tilde{\phi}_n$} the unique smooth solution of a "renormalized equation", \textcolor{black}{analog to \eqref{eqphi} but with a regularized noise and} additional (diverging with $n$) terms. \textcolor{black}{We recall that $c_n(t)=\mathbb{E}(\RSI_n(t)^2)$, and highlighting in red the renormalization terms we write}
\begin{equation}\label{eqphin}\left\{\begin{array}{ll}&\partial_t \tilde{\phi}_n=\Delta \tilde{\phi}_n-(\tilde{\phi}_n^3\textcolor{red}{-3c_n(t)\tilde{\phi}_n})+f_2(t)(\tilde{\phi}_n^2\textcolor{red}{-c_n(t)})+a(t)\tilde{\phi}_n+\sigma\xi_n  \\ & \tilde{\phi}_n(0)=\overline{0} \end{array}\right.\end{equation}
 Then, considering $\tilde{\theta}_n=\tilde{\phi}_n-\RSI_n$, we have
\begin{align*}
    \partial_t \tilde{\theta}_n&=\Delta \tilde{\theta}_n -\left[(\tilde{\theta}_n+\RSI_n)^3-3c_n(\tilde{\theta}_n+\RSI_n)\right]+f_2[(\tilde{\theta}_n+\RSI_n)^2-c_n]+a\tilde{\theta} \\[0,3em]
    &=\Delta \tilde{\theta}_n-\tilde{\theta}_n^3-3\RSI_n \tilde{\theta}_n^2-3(\RSI_n^2-c_n )\tilde{\theta}_n -(\RSI_n^3-3c_n\RSI_n) \\[0,3em] & \ \ \ \ \ \ \ \ +f_2\tilde{\theta}_n^2+2f_2 \tilde{\theta}_n \RSI_n+f_2 (\RSI_n^2-c_n)+a\tilde{\theta}_n \\[0,3em]
    &=\Delta \tilde{\theta}_n-\tilde{\theta}_n^3-3\RSI_n \tilde{\theta}_n^2-3\RSV_n\tilde{\theta}_n -\RSW_n \\[0,3em] & \ \ \ \ \ \ \ \ +f_2\tilde{\theta}_n^2+2f_2 \tilde{\theta}_n \RSI_n+f_2\RSV_n+a\tilde{\theta}_n.
\end{align*}
\textcolor{black}{We here see that the renormalization terms in \eqref{eqphin} can be paired with the powers of $\RSI_n$, that are hence renormalized.} Contrary to \eqref{eqphin}, all the terms in \textcolor{black}{the} equation \textcolor{black}{above} admit a limit when $n$ goes to infinity. \textcolor{black}{Therefore, denoting $\tilde{\phi}=\lim_{n\to+\infty} \tilde{\phi}_n$, could be written $\tilde{\phi}=\RSI+\tilde{\theta}$ where} $\tilde{\theta}:=\lim_{n\to+\infty}\tilde{\theta}_n$ \textcolor{black}{verifies the explicit equation}
$$(\partial_t-\Delta-a(t))\tilde{\theta}=-\tilde{\theta}^3-3\RSI \tilde{\theta}^2-3\RSV \tilde{\theta} -\RSW+f_2(t)\tilde{\theta}^2+2f_2(t) \tilde{\theta} \RSI+f_2(t) \RSV.$$
\textcolor{black}{Since the less regular term on the right-hand side is $\RSW\in\mathcal{C}_{\mathfrak{s}}^{-\frac{3}{2}-\varepsilon}$ for all $\varepsilon>0$, Schauder's theory implies that $\tilde{\theta}$ is of Hölder regularity $\frac{1}{2}-\varepsilon$ in space, and $\tilde{\theta}^2$ and $\tilde{\theta}^3$ are therefore well-defined.} We \textcolor{black}{however} face a new difficulty \textcolor{black}{because $\RSV$} is in $\mathcal{C}_{\mathfrak{s}}^{-1-\varepsilon}$, making the product $\RSV\tilde{\theta}$ undefined.\\

\textcolor{black}{We cannot solve this problem without lengthy computations, and the technical details are therefore left for Section \ref{ce}. In a nutshell, we} use a similar method as before by considering a function $\theta_n$ solution to an equation similar to the one of $\tilde{\theta}_n$ but with new renormalization terms involving the second renormalization constant $\tilde{c}_n$. \textcolor{black}{We then look} at the equation solved by $u_n=\theta_n+\RSIW_n$ before taking the limit (we here use that $(\partial_t-\Delta-a(t))\RSIW=\RSW$ to make the irregular term $\RSW$ disappear).\\

\textcolor{black}{While} $u:=\lim u_n$ \textcolor{black}{does exist, it does not follow a standard equation. Indeed, the diverging in $n$ terms in the equation of $u_n$ cannot be used directly to renormalize symbols and do not therefore disappear.} To solve this last difficulty, \textcolor{black}{we decompose $u$ as a sum of three terms and then use the tools of paracontrolled calculs, following the method outlined in \cite{MW}}. \textcolor{black}{In the end,} we have the \textcolor{black}{subsequent} Lemma where we finally have a well-posed equation with only well-defined terms:
\begin{lemma} \label{defvw}
    \textcolor{black}{Let us consider the following equation:}
    \begin{equation}
\label{e:eqvw}
\left\{
\begin{array}{lll}
(\partial_t - \Delta-a(t)) v & = & F(v+w)\\
(\partial_t - \Delta-a(t)) w & = & G(v,w)
\end{array}
\right.
\end{equation}
where $F:\R\to \R$ and $G:\R\times \R\to \R$ are two random functions \textcolor{black}{whose expressions are given in \eqref{expFG}}. \textcolor{black}{For $(v_0,w_0)=(0,0)$, this system admits a unique solution $(v,w)$, and we have $u=v+w+3I(\RSWW)$.} \textcolor{black}{Besides, there exists $\delta>0$, such that} $v\in \mathcal{C}([0,T],\mathcal{C}^{1-2\varepsilon})\cap \mathcal{C}^{\frac{1}{8}+\delta}([0,T],L^{\infty})$ and $w\in \mathcal{C}([0,T],\mathcal{C}^{\frac{3}{2}-2\varepsilon})\cap \mathcal{C}^{\frac{1}{8}+\delta}([0,T],L^{\infty})$. 
    
\end{lemma}

\begin{remark}
    The term $+3I(\RSWW)$ in the decomposition of $u$ is not usually met when dealing with $(\Phi_3^4)$ and \textcolor{black}{does not appear} in \cite{MW}. We subtracted it from the $w$ of \cite{MW} because it has a component in the Wiener chaos of order $1$, which creates technical difficulties \textcolor{black}{when computing specific bounds} in some proofs (see Remark \ref{ts}). 
\end{remark}

\textcolor{black}{Now that we only have fully defined terms, we can finally explain what the renormalization of \eqref{eqphi} means in practice.} We added a first renormalization term when we tried to define $\tilde{\theta}$ and a second one we tried to define $u$. \textcolor{black}{To get a fully renormalized equation, we therefore have to sum up these two terms, that is to say consider a function $\phi_n$ solution of}
\begin{equation}\label{eqphir}\left\{\begin{array}{ll}&\partial_t \phi_n=\Delta \phi_n-\phi_n^3+f_2 \phi_n^2+a\phi_n \textcolor{red}{+(3c_n-18\tilde{c}_n)\phi_n-f_2c_n}+\sigma\xi_n  \\ & \phi(0)=\overline{0} \end{array}\right.\end{equation}
where we \textcolor{black}{recall} that $c_n(t)\asymp n$ and $\tilde{c}_n(t)\asymp \ln(n)$ \textcolor{black}{for $t>0$}. We know from the paragraphs above that $\phi=\lim_{n\to \infty} \phi_n$ exists, and it is actually the function we consider when we talk about the solution to \eqref{eqphi}. \textcolor{black}{We furthermore have the decomposition $\phi=\RSI-\RSIW+3I(\RSWW)+(v+w)$.}\\

\textcolor{black}{Before stating the main result of this article, we give two final definitions.}\\

\textcolor{black}{First for $A>0$ and $f:[0,A]\mapsto \mathcal{C}^{\beta}$, let $[f]_{\beta,\gamma}:=\sup_{0\le s<t\le A} \frac{\|f(t)-f(s)\|_{\mathcal{C}^{\beta}}}{|t-s|^{\gamma}}$ denote the temporal Hölder-constant of $f$. The definition of $[f]_{\beta,\gamma}$ depends of the domain of $f$ considered, and we should formally write $[f]_{\beta,\gamma,A}$. However, since we will mainly consider the case $A=T$, we will not specify the $A$ dependency from now on, and we will instead say it explicitly when we consider another domain.}\\

\textcolor{black}{Second, for $v\in \mathcal{C}([0,A],\mathcal{C}^{1-2\varepsilon})\cap \mathcal{C}^{\frac{1}{8}}([0,A],L^{\infty})$ and $w\in\mathcal{C}([0,A],\mathcal{C}^{\frac{3}{2}-2\varepsilon})\cap \mathcal{C}^{\frac{1}{8}}([0,A],L^{\infty})$ and $t_0\le A$ we denote}
\begin{align*}&\textcolor{black}{\|(v,w)\|_{\Xi,t_0}} \\&\textcolor{black}{:=\max(\sup_{s\in [0,t_0]} \|v(s)\|_{\mathcal{C}^{1-2\varepsilon}},\sup_{s\in[0,t_0]}\|w(s)\|_{\mathcal{C}^{\frac{3}{2}-2\varepsilon}},[v_{|[0,t_0]}]_{0,\frac{1}{8}},[w_{|[0,t_0]}]_{0,\frac{1}{8}})} 
\end{align*}
\textcolor{black}{where $v_{[0,t_{0}]}$ and $w_{|[0,t_0]}$ denote the restriction of $v$ and $w$ to $[0,t_0]$. We prove easily that $\|(v,w)\|_{\Xi,A}$ is a complete norm.}\\

We can \textcolor{black}{now} state the main result of this article.
\begin{theorem} \label{thp}
    \textcolor{black}{Let us denote $m=\|f_2\|_{\infty}+\|f_2'\|_{\infty}$.} For $\varepsilon>0$ and $\lambda\in(0,\frac{\varepsilon}{3}\wedge 1)$, there exists $h_0\textcolor{black}{:=h_0(\lambda,a_+,a_{-},m)\in (0,1)}$, and $C:=C(\lambda,\textcolor{black}{a_+,a_{-}}),D:=D(\lambda,\textcolor{black}{a_+,a_{-}})>0$ independent of $T$ such that for all $h\in (0,h_0)$ we have
    $$\mathbb{P}\left( \textcolor{black}{\|(v,w)\|_{\Xi,T}}> h\right)\le D \exp\left(-\frac{C}{\textcolor{black}{\max(T^{\lambda},T^{\frac{\lambda}{5}})}} \frac{h^2}{\sigma^2}\right).$$
\end{theorem} 
Here we recall that for $\varphi$ solution of the renormalized version of \eqref{mne} \textcolor{black}{(easily deduced from \eqref{eqphir})}, we have $\varphi=\overline{\phi}+\RSI-\RSIW+3I(\RSWW)+(v+w)$.\\

\textcolor{black}{Combining Theorem \ref{cs} for the symbols and Theorem \ref{thp} for $(v,w)$ we have the following key corollary which is the result stated in the abstract:}
\begin{corollary}\label{corp}
    \textcolor{black}{For $\varepsilon>0$ and $\lambda\in(0,\frac{\varepsilon}{3}\wedge 1)$, there exists an $h_0':=h_0'(\lambda,a_+,a_{-},m)\in(0,1)$, and $C':=C'(\lambda,a_+,a_{-}),D':=D'(\lambda,a_+,a_{-})>0$ independent of $T$ such that for all $h\in (0,h_0)$ we have} 
    $$\textcolor{black}{\mathbb{P}\left(\sup_{t\in [0,T]}\|(\varphi-\overline{\phi})(t)\|_{\mathcal{C}^{-\frac{1}{2}-\varepsilon}}> h\right)\le D' \exp\left(-\frac{C'}{\max(T^{\lambda},T^{\frac{\lambda}{5}})} \frac{h^2}{\sigma^2}\right).}$$
\end{corollary}

\begin{remark}\label{ga3}
    \textcolor{black}{Let us go back to \eqref{mne} and assume that $\textcolor{black}{R}(t,\varphi)=a_3(t)\varphi^3+a_2(t) \varphi^2+a_1(t)\varphi+a_0(t)$ where $a_0,a_1,a_2,a_3:[0,T]\to \R$ are differentiable and $a_3(t)<0$ for all $t\in[0,T]$ (and is therefore bounded away from $0$). As we said in Remark \ref{mg}, all the results of this article hold for such a general $\textcolor{black}{R}$, and we will now explain why.} \textcolor{black}{Writing} $\varphi=b(t)\psi$, we get
\begin{align*}
    b\partial_t \psi&=b\Delta \psi+a_3b^3\psi^3+a_2b^2\psi^2+(a_1b-b')\psi+a_0+\sigma\xi
\end{align*}
so, taking $b=\frac{1}{\sqrt{-a_3}}$ which is differentiable, and dividing the equation by $b(t)$ we get \textcolor{black}{essentially a rewriting of \eqref{mne}}
$$\left\{\begin{array}{ll}
     &\partial_t \psi=\Delta \psi-\psi^3+b_2(t)\psi^2+b_1(t)\psi+b_0(t)+\sigma\xi^b  \\
     & \psi(0)=\frac{1}{b(0)}\varphi_0
\end{array}\right.$$
where $b_2(t)=a_2(t)b(t)$, $b_1(t)=\frac{a_1(t)b(t)-b'(t)}{b(t)}$, $b_0(t)=\frac{a_0(t)}{b(t)}$ and $\xi^b=\frac{1}{b} \xi$. This new gaussian white noise is not homogeneous in time, but since $b$ is positive, bounded away from $0$ and continuous on $[0,T]$, $\xi^b$ has essentially the same properties as $\xi$. For instance, for $\varphi_1,\varphi_2\in L^2(\R\times \mathbb{T}^3)$ of support included in $[0,T]\times \mathbb{T}^3$, we have
$$\mathbb{E}((\varphi_1,\xi^b)(\varphi_2,\xi^b))=\int_{[0,T]\times\mathbb{\mathbb{T}^3}} \frac{1}{b(t)^2}\varphi_1(t,x)\varphi_2(t,x)\mathrm{d}t\mathrm{d}x.$$
\textcolor{black}{Since our concentration results hold for $\psi$, they hold for $\varphi$.} 
\end{remark}
\begin{remark}\label{hypa}
    \textcolor{black}{All our theorems would still hold if we had not assumed $a<0$, but the constants in them would not be of the form $T^{-\lambda}$ anymore}. Indeed, if we assume that $a$ may take positive values, we have to bound terms like $\exp(\int_s^t a(u)\mathrm{d}u)$ by $\exp(T\times \sup a)$, and therefore all \textcolor{black}{the} constants appearing in our concentration results will exponentially depend of $T$, making them far less relevant for large times $T$.
\end{remark}

\subsection{An application : \textcolor{black}{a} non-autonomous $(\Phi_3^4)$ model \label{ex}}
We consider the equation
\begin{equation}\left\{\begin{array}{ll}&\partial_t\varphi=\Delta \varphi-\varphi^3+\gamma(t)\varphi+\sigma\xi \\ &\varphi(0)=\varphi_0 \label{phi43} \end{array} \right.\end{equation}
which is \eqref{mne} with $\textcolor{black}{R}(t,\varphi)=-\varphi^3+\gamma(t)\varphi$. \textcolor{black}{We} can decompose a solution $\varphi$ of \eqref{phi43} as
$$\varphi=\overline{\phi}+\phi$$
where $\overline{\phi}$ is the solution of the deterministic equation
\begin{equation}\left\{\begin{array}{ll}&\partial_t\overline{\phi}=\Delta \overline{\phi}-\overline{\phi}^3+\gamma(t)\overline{\phi} \\
&\overline{\phi}(0)=\varphi_0\label{phisb} \end{array} \right.\end{equation}
and $\phi$ is the solution of
$$\left\{ \begin{array}{ll} &\partial_t \phi=\Delta \phi+\textcolor{black}{\tilde{R}}(t,\phi)+\sigma\xi \\ 
&\phi(0)=0\end{array}\right.$$
with $\textcolor{black}{\tilde{R}}(t,\phi)=-\phi^3+f_2(t)\phi^2+f_1(t)\phi$.
We look for \textcolor{black}{a} stable solution \textcolor{black}{of} \eqref{phisb} that is constant in space, that is to say a \textcolor{black}{function} $\overline{\phi}:[0,T]\to \R$ such that \begin{equation} \label{ece}
    \left\{ \begin{array}{ll}
         \partial_t \overline{\phi}(t) = -\overline{\phi}(t)^3+\gamma(t)\overline{\phi}(t)  \\
        \gamma(t)-3\overline{\phi}(t)^2<0   
    \end{array} \right. \ \ \ \   \text{for all $t\in[0,T]$} 
\end{equation} \textcolor{black}{As was said in the Assumption, the inequality in \eqref{ece} is equivalent to} $a:=f_1<0$. There are two noteworthy particular cases:

\begin{itemize}
    \item[$\bullet$] If $\gamma(t)<0$ for all $t\in[0,T]$, the condition $\gamma(t)-3\overline{\phi}(t)^2<0$ is always verified. Therefore, for any constant in space initial condition, we have that $\overline{\phi}$ is a stable solution.
    \item[$\bullet$] If $\gamma(t)>0$ for all $t\in[0,T]$ and $\gamma$ is decreasing, we can check that $\phi_+(t)=\sqrt{\gamma(t)/3}$ is a strict subsolution to \eqref{phisb} and $\phi_{-}(t)=-\sqrt{\gamma(t)/3}$ is a strict supersolution. Therefore, if $\overline{\phi}(0)>\phi_+(0)$ (respectively $\overline{\phi}(0)<\phi_{-}(0)$), we have that $\overline{\phi}(t)>\phi_{+}(t)$ for all $t\in [0,T]$ (respectively $\overline{\phi}(t)<\phi_{-}(t)$ for all $t\in [0,T]$), which means that $\gamma(t)-3\overline{\phi}(t)^2<0$. Eventually, for any constant in space initial condition taken outside of $[-\sqrt{\gamma(0)/3},\sqrt{\gamma(0)/3}]$, we have that $\overline{\phi}$ is a stable solution.
\end{itemize}
\textcolor{black}{We can therefore invoke Corollary \ref{corp}, which tells us that if $\sigma\le h_0$, we have that $\sup_{t\in[0,T]}\|(\varphi-\overline{\phi})(t)\|_{\mathcal{C}^{-\frac{1}{2}-\varepsilon}}$ is at most of order $\sigma$ with a high probability.}

\section{Proof of the main theorems}

\textcolor{black}{In order to prove the main concentration results, we will combine several technical tools found in our different references. Here are some insights on the mathematical ideas behind the proof.}
\begin{itemize}
    \item \textcolor{black}{First}, a key idea behind the resolution of $(\Phi_3^4)$ using paracontrolled distributions is to to regularize previously undefined products by trading-off temporal regularity for spatial regularity in some equations. Therefore, random temporal Hölder constants appear naturally in our computations, and to get concentration results \textcolor{black}{on the} solutions of $(\Phi_3^4)$ we first need concentration results on \textcolor{black}{the} random temporal Hölder constants \textcolor{black}{of symbols}. Computing the moments of a Hölder constant is however a non-trivial task, and we cannot directly use the results found of \cite{MWX} to do this. To overcome this difficulty, we invoke the Garsia-Rodemich-Rumsey inequality proven in \cite{GRR} to bound the supremum in the definition of the Hölder constant with an integral involving terms we can control. Besides, concentration results on $\sup_{t\in[0,T]}\|\tau(t)\|$ for $\tau$ a symbol are directly implied by concentration results on the temporal Hölder constant of $\tau$ (since we know the initial condition), making the last ones even more useful in our context.
    \item  \textcolor{black}{Second}, we underscore that the theory of paracontrolled distributions is formulated in the framework of Besov spaces, and that we will identify the Hölder space $\mathcal{C}^\alpha$ with the Besov space $\mathcal{B}_{\infty,\infty}^{\alpha}$ (see Appendix). While this definition \textcolor{black}{of} Hölder spaces seems complicated at first glance\textcolor{black}{,} since we decompose a function $f$ into an infinite sum of "blocks" $\delta_k f$, it is particularly relevant when $f$ is a random function we want to get concentration results on. Indeed, to control the tail of a random variable one can control its moments, and Proposition \ref{critere} proves that we can bound the $p$-th moment of $\|f\|_{\mathcal{C}^{\alpha}}$ with bounds on the $p$-th moments of the $\|\delta_k f\|_{L^p}$. The blocks being always smooth while $f$ is often only a distribution, computations are simpler with blocks than with $f$. 
    \item \textcolor{black}{Third}, all the symbols that we use in this article are be expressed with stochastic integrals only involving independent copies of the stochastic white noise $\xi$. This allows us to use the theory of Wiener chaos and especially Nelson's inequality (Proposition \ref{nelson}) to control the $p$-th moment of key random variables with bounds on their second moments. These kinds of bounds are precisely the ones proven in \cite{MWX}.  
    \item \textcolor{black}{Fourth, once Theorem \ref{cs} is proven, we} combine technical Lemmas found in the Appendix of \cite{MW} with the method developed in \cite{BN2} to transform \textcolor{black}{this} concentration result on symbols into concentration results on the solution of an equation involving these symbols. \textcolor{black}{This part is longer and more technical than the previous ones, but it is not conceptually complex.}
\end{itemize}

\subsection{A note on bounds from \cite{MWX} \label{borneMWX}}
\textcolor{black}{For $\tau \in \mathcal{T}$ and $t\ge 0$, let us denote $(\hat{\tau}(t,\omega))_{\omega\in \mathbb{Z}^d}$ the Fourier coefficients of $\tau(t)$. In \cite{MWX}, the authors prove for all $\tau\in \mathcal{T}$ the inequalities \cite{MWX}[(42)]}
$$\mathbb{E}(|\hat{\tau}(t,\omega)|^2)\le C(1+|\omega|)^{-d-2|\tau|}$$
\textcolor{black}{and \cite{MWX}[(44)]}
$$\mathbb{E}(|\hat{\tau}(t,\omega)-\hat{\tau}(s,\omega)|^2)\le C|t-s|^{\lambda}(1+|\omega|)^{-d-2|\tau|+2\lambda}$$
\textcolor{black}{for all $\lambda\in(0,1)$, $\omega\in \mathbb{Z}^d$ and $0<|t-s|<1$.} 

\begin{remark} \label{remY}
    \textcolor{black}{In fact, these two inequalities are} not explicitly proven \textcolor{black}{for $\RSY$} in \cite{MWX}, but \textcolor{black}{they are easily derived from} the proof of the inequalities for $\RSV$ \textcolor{black}{(see \cite{MWX}[(75)] where a bound on $\RSIW$ is derived from a bound on $\RSW$)}.
\end{remark}

\textcolor{black}{In this section, we prove that these inequalities still hold in our setting, that is a bit different from the one of \cite{MWX}. We will need them to prove \eqref{est}, a technical estimate used in the proof of Theorem \ref{cs}.} \\

\textcolor{black}{Let us first explain how our setting differs from the one of \cite{MWX}}. In \cite{MWX} the authors consider formally for $P_{t-s}=e^{-(t-s)}e^{(t-s)\Delta}$
$$\tilde{\RSI}(t)=\int_{-\infty}^t P_{t-s}(\xi(s))\mathrm{d}s,$$
that is an "ancient solution" to the homogeneous in time stochastic heat equation. Here, we consider for $P_{\textcolor{black}{s,t}}=e^{\int_s^t a(u)\mathrm{d}u}e^{(t-s)\Delta}$
$$\RSI(t)=\int_0^t P_{\textcolor{black}{s,t}}(\xi(s))\mathrm{d}s,$$
that is a solution to the non-autonomous heat equation \textcolor{black}{with trivial initiation condition}. We can however prove that bounds of the type \cite[(42)\&(44)]{MWX} are still valid in our setting, and that constants in them are independent of $T$. It would be tedious to prove this for all seven symbols, so we will prove it for $\RSI$. The reader can convince themself that those bounds work for all other symbols.\\

\textcolor{black}{We will now introduce two key notations. First, we denote $\hat{P}_{s,t}$ the Fourier transform of the heat kernel, that is for $\omega\in \Z^3$}
$$\textcolor{black}{\hat{P}_{s,t}(\omega)=e^{\int_s^t a(u)\mathrm{d}u}e^{-(t-s)4\pi^2|\omega|^2}.}$$\textcolor{black}{Second, for $e_{\omega}(x)=e^{2i\pi \omega\cdot x}$, we denote $(W(\omega,t))_{t\ge 0}$ the complex Brownian motion $(\xi(\mathbb{1}_{[0,t]}e_{\omega}))_{t\ge 0}$.} \textcolor{black}{Using these notations, we can write} for $t\in [0,T]$ and $\omega\in \Z^3$
\begin{align*}\hat{\RSI}(t,\omega)&=\int_0^t \hat{P}_{s,t}\mathrm{d}W(s,\omega) \\
&=\int_0^t e^{\alpha(t,s)}e^{-(t-s)4\pi^2|\omega|^2}\mathrm{d}W(s,\omega)
\end{align*}
Then, writing $a_{+}=-\sup_{[0,T]} a>0$, we have that $a+a_{+}\le 0$ \textcolor{black}{on $[0,T]$} and therefore
\begin{align*}\mathbb{E}[|\hat{\RSI}(t,\omega)|^2 ]&=\int_0^t \hat{P}_{\textcolor{black}{s,t}}\textcolor{black}{(\omega)}^2\mathrm{d}s \\
&=\int_0^t e^{2 \int_s^t [a(u)+a_{+}]\mathrm{d}u}e^{-2(t-s)(a_{+}+4\pi^2|\omega|^2)}\mathrm{d}s \\
&\le \int_0^t e^{-2(t-s)(a_{+}+4\pi^2|\omega|^2)} \mathrm{ds} \\
&\le \frac{1}{2(a_{+}+4\pi^2|\omega|^2)}
\end{align*}
This is exactly the estimate \cite[(42)]{MWX} \textcolor{black}{and we underscore that it is} independent of $T$. If we now look at temporal differences\textcolor{black}{, recalling that we denote $\alpha(t,s)=\int_{s}^t a(u)\mathrm{d}u$,} we have for $0\le s\le t\le T$
\begin{align*}
    &\mathbb{E}[|\hat{\RSI}(t,\omega)-\hat{\RSI}(s,\omega)|^2] \\ &\le\int_0^s |\hat{P}_{u,t}\textcolor{black}{(\omega)}-\hat{P}_{u,s}\textcolor{black}{(\omega)}|^2 \mathrm{d}u\textcolor{black}{+\int_s^t |\hat{P}_{u,t}(\omega)|^2\mathrm{d}u} \\
    &\le \int_0^s |e^{\alpha(t,u)}e^{-(t-u)4\pi^2|\omega|^2}-e^{\alpha(s,u)}e^{-(s-u)4\pi^2|\omega|^2}|^2 \mathrm{d}u \\
    & \ \ \ \ \ \ \ \textcolor{black}{+\int_s^t e^{2\alpha(t,u)}e^{-2(t-u)4\pi^2|\omega|^2}\mathrm{d}u} \\
    &\le 2\int_0^s |e^{\alpha(s,u)}e^{-(t-u)4\pi^2|\omega|^2}-e^{\alpha(s,u)}e^{-(s-u)4\pi^2|\omega|^2}|^2 \mathrm{d}u \\
    & \ \ \ \ \ \ \ +2\int_0^s |(e^{\alpha(t,s)}-1)e^{\alpha(s,u)}e^{-(t-u)4\pi^2|\omega|^2}|^2 \mathrm{d}u \\
    & \ \ \ \ \ \ \ \ \ \ \ \ \ \ +\int_s^t e^{-2a_{+}(t-u)}e^{-2(t-u)4\pi^2|\omega|^2}\mathrm{d}u\\
    &\le 2\int_0^s \textcolor{black}{e^{-2(s-u)a_{+}}} |e^{-(t-u)4\pi^2|\omega|^2}-e^{-(s-u)4\pi^2|\omega|^2}|^2 \mathrm{d}u \\
    & \ \ \ \ \ \ \ +2\int_0^s |(e^{\alpha(t,s)}-1)e^{-(s-u)a_{+}}e^{-(t-u)4\pi^2|\omega|^2}|^2  \mathrm{d}u\\
    & \ \ \ \ \ \ \ \ \ \ \ \ \ \ +\int_s^t e^{-2(t-u)(a_{+}+4\pi^2|\omega|^2)}\mathrm{d}u \\
    &\le 2|e^{-(t-s)4\pi^2|\omega|^2}-1|^2\int_0^s e^{-2(s-u)(a_{+}+4\pi^2|\omega|^2)} \mathrm{d}u\\
     & \ \ \ \ \ \ \ +2|e^{\alpha(t,s)}-1|^2\int_0^s e^{-2(t-s)4\pi^2|\omega|^2}e^{-2(s-u)(a_{+}+4\pi^2|\omega|^2)}\mathrm{d}u \\
     & \ \ \ \ \ \ \ \ \ \ \ \ \ \ +\frac{1}{2(a_{+}+4\pi^2|\omega|^2)}[1-e^{-2(t-s)(a_{+}+4\pi^2|\omega|^2)}]
\end{align*}
We then use the standard inequalities 
\begin{equation} \label{inc}|1-e^{-x}|\le 1\wedge x \textcolor{black}{\le 1\wedge x^{\lambda}}\le 1\wedge x^{\frac{\lambda}{2}}\end{equation} for $x\ge 0$ and $\lambda\in (0,1)$. Denoting $a_{-}=-\inf a$ we have
\begin{align*}
   &\mathbb{E}[|\hat{\RSI}(t,\omega)-\hat{\RSI}(s,\omega)|^2] \\ &\le 2(1\wedge (t-s)^{\lambda}(4\pi^2|\omega|^2)^{\lambda})\int_0^s e^{-2(s-u)(a_{+}+4\pi^2|\omega|^2)}\mathrm{d}u\\
     & \ \ \ \ \ \ \ +2(1\wedge (a_{-})^{\lambda}(t-s)^{\lambda})\int_0^s e^{-2(s-u)(a_{+}+4\pi^2|\omega|^2)}\mathrm{d}u \\
     &  \ \ \ \ \ \ \ \ \ \ \ \ \ \ +(1\wedge 2^{\lambda}(t-s)^{\lambda}(a_{+}+4\pi^2|\omega|^2)^{\lambda})\frac{1}{2(a_{+}+4\pi^2|\omega|^2)} \\
     &\le 2(t-s)^{\lambda}((4\pi^2|\omega|^2)^{\lambda}+(a_{-})^{\lambda}+(a_{+}+4\pi^2|\omega|^2)^{\lambda})\frac{1}{2(a_{+}+4\pi^2|\omega|^2)} \\
     &\le C_1(t-s)^{\lambda}\langle \omega \rangle^{-2+2\lambda}
\end{align*}
where $\langle \omega \rangle=\sqrt{1+|\omega|^2}$ and $C_1:=C(a_{+},a_{-})$ is independent of $\lambda$ since we can bound terms like $(a_{-})^{\lambda}$ by $(1+a_{-})$. This is exactly the estimate \cite[(44)]{MWX} for $\RSI$. In fact, it is better than \cite[(44)]{MWX}, since in \cite{MWX} they ask for $|t-s|$ to be less than $1$, while we have an inequality that does not depend of $T$ and is uniform in $0\le s\le t\le T$.

\subsection{Proof of Theorem \ref{cs}}

\textcolor{black}{We first state a technical Lemma.}

\begin{proposition}\label{ch}
    Let $\tau\in \mathcal{T}$, and $\beta<|\tau|-\lambda$ for $\lambda\in (0,1)$. \textcolor{black}{Denoting $\gamma_0:=(\frac{\lambda}{2}-|\frac{\ln(2)}{\ln(T)}|)\vee\frac{\lambda}{4}\in[\frac{\lambda}{4},\frac{\lambda}{2})$, for all $\gamma\in (\gamma_0,\frac{\lambda}{2})$} there exist $\ell_{\tau}:=\ell_{\tau}(\lambda,\textcolor{black}{\beta,a_{+},a_{-}})$ and $m_{\tau}:=m_{\tau}(\lambda,\textcolor{black}{\beta,\gamma,a_{+},a_{-}})$ independent of $T$, \textcolor{black}{such that for all $k\le n_{\tau}$ we have} $$\mathbb{P}\left([\Pi_k\tau]_{\beta,\gamma}>h^{k}\right)\le m_{\tau} \exp\left(-\ell_{\tau}\frac{h^2}{\sigma^2}\right).$$
\end{proposition}
The proof relies heavily on the Garsia-Rodemich-Rumsey Lemma (see \cite[Lemma 1.1]{GRR}).
\begin{lemma}\label{GRR}
    Let $(\mathcal{B},|\cdot|)$ be a Banach space, $f:[0,1]\mapsto \mathcal{B}$ a continuous function, $\Psi:\R_+\to\R_+$ a strictly increasing function with $\Psi(+\infty)=+\infty$ and $p:[0,1]\to[0,1]$ strictly increasing with $p(0)=0$, such that
    $$\textcolor{black}{B:=}\int_0^1 \int_0^1 \Psi\left[\frac{|f(x)-f(y)|}{p(|x-y|)}\right] \mathrm{d}x\mathrm{d}y<+\infty.$$
    Then, for all $s,t\in [0,1]$, we have
    $$|f(t)-f(s)|\le 8 \int_0^{|t-s|} \Psi^{-1}\left(\frac{4B}{u^2}\right)\mathrm{d}p(u)$$
\end{lemma}
\begin{proof} (Proposition \ref{ch}) \\\
\textcolor{black}{Fix $\lambda\in(0,1)$ and $\gamma\in(\gamma_0,\frac{\lambda}{2})$. Let $\gamma'=\frac{1}{2}(\gamma+\frac{\lambda}{2})$, and let us take} $\Psi(u)=|u|^p$ and $p(u)=|u|^{\gamma'+\frac{1}{p}}$ in Lemma \ref{GRR} \textcolor{black}{for} $(\mathcal{B},|\cdot|)$ a general Banach space. Using the notations of the Lemma, we then get that \textcolor{black}{for any $f:[0,1]\to \mathcal{B}$ continuous}
\begin{align*}
    |f(t)-f(s)|&\le 8 \int_0^{|t-s|} \frac{(4B)^{\frac{1}{p}}}{u^\frac{2}{p}}(\gamma'+\frac{1}{p})u^{\gamma'-1+\frac{1}{p}}\mathrm{d}u \\
    &\le 8\cdot 4^{\frac{1}{p}}(\gamma'+\frac{1}{p}) B^{\frac{1}{p}}\int_0^{|t-s|} u^{\gamma'-\frac{1}{p}-1}\mathrm{d}u  \\
    &\le8\cdot 4^{\frac{1}{p}}\frac{\gamma'+\frac{1}{p}}{\gamma'-\frac{1}{p}}|t-s|^{\gamma'-\frac{1}{p}}\left(\int_0^1 \int_0^1 \frac{|f(x)-f(y)|^p}{|x-y|^{\gamma' p+1}}\mathrm{d}x\mathrm{d}y\right)^{\frac{1}{p}}
\end{align*}
The only term involving $t,s$ on the right-hand side is $|t-s|^{\gamma'-\frac{1}{p}}$ so we have
$$\sup_{0\le s<t\le 1}\frac{|f(t)-f(s)|^p}{|t-s|^{\gamma' p -1}}\le \left(8\cdot 4^{\frac{1}{p}} \frac{\gamma'+\frac{1}{p}}{\gamma'-\frac{1}{p}}\right)^p \int_0^1 \int_0^1 \frac{|f(x)-f(y)|^p}{|x-y|^{\gamma' p+1}}\mathrm{d}x\mathrm{d}y.$$
This inequality is noteworthy, because it says that a control on the integral of a quantity implies a control on a closely related supremum. Let us now apply it to the case $f(t)=\tau(T\cdot t)$ where we take $(\mathcal{B,|\cdot|})=(\mathcal{C}^{\beta},\|\cdot\|_{\mathcal{C}^{\beta}})$. We have
\begin{align*}&\sup_{0\le s<t\le 1}\frac{\|\tau(t\cdot T)-\tau(s\cdot T)\|_{\mathcal{C}^{\beta}}^p}{|t-s|^{\gamma' p -1}} \\ &\le \left(8\cdot 4^{\frac{1}{p}} \frac{\gamma'+\frac{1}{p}}{\gamma'-\frac{1}{p}}\right)^p (\int_0^1 \int_0^1 \frac{\|\tau(T\cdot x)-\tau(T\cdot y)\|_{\mathcal{C}^{\beta}}^p}{|x-y|^{\gamma' p+1}}\mathrm{d}x\mathrm{d}y)\end{align*}
therefore, doing a change of variable we get
\begin{align*}&\sup_{0\le s<t\le T}\frac{\|\tau(t)-\tau(s)\|_{\mathcal{C}^{\beta}}^p}{|t-s|^{\gamma'p-1}} T^{\gamma' p-1} \\ &\le \left(8\cdot 4^{\frac{1}{p}} \frac{\gamma'+\frac{1}{p}}{\gamma'-\frac{1}{p}}\right)^p (\int_0^T \int_0^T \frac{\|\tau(x)-\tau(y)\|_{\mathcal{C}^{\beta}}^p}{|x-y|^{\gamma' p+1}}\mathrm{d}x\mathrm{d}y)\frac{1}{T^2} T^{\gamma'p+1}\end{align*}
and we can simplify the terms in $T$
\begin{equation}\sup_{0\le s<t\le T}\frac{\|\tau(t)-\tau(s)\|_{\mathcal{C}^{\beta}}^p}{|t-s|^{\gamma'p-1}} \le \left(8\cdot 4^{\frac{1}{p}} \frac{\gamma'+\frac{1}{p}}{\gamma'-\frac{1}{p}}\right)^p \int_0^T \int_0^T \frac{\|\tau(x)-\tau(y)\|_{\mathcal{C}^{\beta}}^p}{|x-y|^{\gamma' p+1}}\mathrm{d}x\mathrm{d}y\label{sup}.\end{equation}
\textcolor{black}{By replacing} $\gamma' p-1$ by $\gamma p$, \eqref{sup} implies that \textcolor{black}{it is enough to estimate} $\mathbb{E}\frac{\|\tau(t)-\tau(s)\|_{\mathcal{C}^{\beta}}^p}{|x-y|^{\gamma' p+1}}$ \textcolor{black}{to bound the $p$-th moment of $[\tau]_{\beta,\gamma}$.}\\

To get these estimates, we will use that $\tau:t\mapsto \tau(t)$ is in the $n_{\tau}$-th inhomogeneous Wiener chaos (and therefore $\delta_k \tau(t)$ also) where $n_{\tau}$ is the number of leaves of $\tau$. In other words, we have that $(\tau(t),\phi)\in \mathcal{H}_{\le n_{\tau}}$ for all smooth functions $\phi$ on $\mathbb{T}^3$ and $t\in \R$, with furthermore an expectation equal to $0$. Since $\delta_k \tau(t)$ is a continuous function \textcolor{black}{in space} for all $k\ge -1$ and $t\in [0,T]$, we have $\delta_k\tau(t,x)\in \mathcal{H}_{\le n_{\tau}}$ for all $(t,x)\in [0,T]\times \mathbb{T}^3$. Hence \textcolor{black}{by} Fubini's theorem and Nelson's estimate (Proposition \ref{nelson}), we have
\begin{align*}
    \mathbb{E}[\|\delta_k\tau(t)-\delta_k \tau(s)\|_{L^p}^p]&\le \sup_{z\in \mathbb{T}^3} \mathbb{E}[|\delta_k \tau(t,z)-\delta_k \tau(s,z)|^p] \\
    &\le \sup_{z\in \mathbb{T}^3} C_{n_{\tau}}^p(p-1)^{\frac{n_{\tau}p}{2}} \left(\mathbb{E}[|\delta_k\tau(t,z)-\delta_k \tau(s,z)|^2]\right)^{\frac{p}{2}}.
\end{align*}

Then, we use estimates \textcolor{black}{of the type \cite{MWX}[(42)\&(44)]} \textcolor{black}{(see Section \ref{borneMWX})} and the proof of \cite[Proposition 5] {MWX} to get that for $\alpha=|\tau|-\lambda$ and all $(t,z)\in [0,T]\times \mathbb{T}^3$
\begin{equation}\mathbb{E}[|\delta_k \tau(t,z)-\delta_k \tau(s,z)|^2]\le C_1 |t-s|^{\lambda} 2^{-2k\alpha},\label{est}\end{equation}
so 
$$\mathbb{E}[\|\delta_k\tau(t)-\delta_k \tau(s)\|_{L^p}^p]\le |t-s|^{\frac{\lambda p}{2}} C_{n_{\tau}}^p C_1^{\frac{p}{2}}(p-1)^{\frac{n_{\tau}p}{2}} 2^{-kp\alpha}.$$
Since $\beta<\alpha$, Proposition \ref{critere} gives us that, for $p>\frac{3}{\alpha-\beta}+1$,
\begin{align*}\mathbb{E}[\|\tau(t)-\tau(s)\|_{\mathcal{C}^{\beta}}^p]&\le C_0^p  \sup_{k\ge -1} 2^{kp\alpha}|t-s|^{\frac{\lambda p}{2}} C_{n_{\tau}}^p C_1^{\frac{p}{2}} (p-1)^{\frac{n_{\tau}p}{2}} 2^{-k p \alpha} \\ &\le|t-s|^{\frac{\lambda p}{2}} C_0^p C_{n_{\tau}}^p C_1^{\frac{p}{2}} (p-1)^{\frac{n_{\tau}p}{2}} . \numberthis \label{in2}\end{align*}

Since $\gamma<\gamma'<\frac{\lambda}{2}$ and $0<|t-s|< T$ we have \textcolor{black}{for $p\ge\frac{1}{\frac{\lambda}{2}-\gamma'}=\frac{1}{\gamma'-\gamma}=\frac{2}{\frac{\lambda}{2}-\gamma}$ that}
\begin{equation}\label{cr}T^{\gamma'p+1-\frac{\lambda p}{2}}|t-s|^{\frac{\lambda p}{2}}\le |t-s|^{\gamma'p+1} \ \ \text{   and   } \ \ |t-s|^{\gamma'p-1}\le T^{\gamma' p-1-\gamma p}  |t-s|^{\gamma p}.\end{equation} Choosing $L\ge 8\cdot 4^{\frac{1}{p}} \frac{\frac{\lambda}{2}+\frac{1}{p}}{\frac{\lambda}{4}-\frac{1}{p}} \ge 8\cdot 4^{\frac{1}{p}} \frac{\gamma'+\frac{1}{p}}{\gamma'-\frac{1}{p}}$, \textcolor{black}{we take} $p>\max(\textcolor{black}{\frac{3}{\alpha-\beta}+1},\frac{4}{\lambda}+1,\frac{2}{\frac{\lambda}{2}-\gamma})\ge\max(\textcolor{black}{\frac{3}{\alpha-\beta}+1},\frac{1}{\gamma'}+1,\frac{2}{\frac{\lambda}{2}-\gamma})$, \textcolor{black}{and} we have from \eqref{sup}, \eqref{in2} and \eqref{cr}
\begin{align*}&\mathbb{E}\left[\sup_{\textcolor{black}{0\le}s<t\textcolor{black}{\le T}}\frac{\|\tau(t)-\tau(s)\|_{\mathcal{C}^{\beta}}^p}{|t-s|^{\gamma p}}\right] \\ &\le T^{\gamma'p-1-\gamma p-(\gamma'p+1-\frac{\lambda p}{2})} L^p \int_0^T \int_0^T \mathbb{E}\frac{\|\tau(x)-\tau(y)\|_{\mathcal{C}^{\beta}}^p}{|x-y|^{\frac{\lambda p}{2}}}\mathrm{d}x\mathrm{d}y \\ 
&\le T^{\frac{\lambda p}{2}-\gamma p} L^p C_0^p C_{n_{\tau}}^p C_1^{\frac{p}{2}} (p-1)^{\frac{n_{\tau}p}{2}}. \numberthis \label{m1}
\end{align*}
The left-hand side is exactly $\mathbb{E}([\tau]_{\beta,\gamma}^p)$, and these bounds on the moments of the Hölder constant allow us to control its tail. \textcolor{black}{We now use the definition of $\gamma_0$ given in the statement of Proposition \ref{ch}, that is tailored to get} $T^{\frac{\lambda}{2}-\gamma_0}\le2$ so that $T^{\frac{\lambda p}{2}-\gamma p}\le 2^p$. Finally, using that all considered quantities are positive, we observe that for $v>0$
\begin{align*}
    \mathbb{E}[\exp(v[\tau]_{\beta,\gamma}^{\frac{2}{n_{\tau}}})]&=\mathbb{E}\left[\sum_{k=0}^{+\infty} \frac{v^k [\tau]_{\beta,\gamma}^{\frac{2k}{n_{\tau}}}}{k!}\right] \\
    &\le \sum_{k=0}^{+\infty} \frac{v^k\mathbb{E}([\tau]_{\beta,\gamma}^{\frac{2k}{n_{\tau}}})}{k!} \\
    &\le \sum_{k=0}^{p_0-1} \frac{v^k\mathbb{E}([\tau]_{\beta,\gamma}^{\frac{2k}{n_{\tau}}})}{k!} + \sum_{k=p_0}^{+\infty} \frac{v^k 2^{\frac{2k}{n_{\tau}}} L^\frac{2k}{n_{\tau}} C_0^{\frac{2k}{n_{\tau}}}  C_{n_{\tau}}^\frac{2k}{n_{\tau}}C_1^{\frac{k}{n_{\tau}}} (\frac{2k}{n_{\tau}}-1)^{k}}{k!}
\end{align*}
    where $p_0=\lfloor\max(\frac{3}{\alpha-\beta}+1,\frac{4}{\lambda}+1,\textcolor{black}{\frac{2}{\frac{\lambda}{2}-\gamma}})\rfloor\textcolor{black}{+1}$ so that we can use \eqref{m1}. \\
    
    Since $\mathbb{E}([\tau]_{\beta,\gamma}^{\frac{2p_0}{n_{\tau}}})$ is bounded uniformly in $T$, the first sum is finite and the number of terms in it is independent of $T$, \textcolor{black}{and therefore} it converges to a quantity independent of $T$. For the second one, writing $K_0=2^{\frac{2}{n_{\tau}}}L^{\frac{2}{n_{\tau}}}C_0^{\frac{2}{n_{\tau}}}C_{n_{\tau}}^{\frac{2}{n_{\tau}}} C_1^{\frac{1}{n_{\tau}}}$, we have using Stirling's formula that
    \begin{align*}
        \frac{v^kK_0^k (\frac{2k}{n_{\tau}}-1)^k}{k!}\le \frac{v^kK_0^k(\frac{2}{n_{\tau}})^k k^k}{k!} &\sim (vK_0\frac{2}{n_{\tau}})^k \frac{1}{\sqrt{2\pi k}}(\frac{k}{k})^k e^k \\ &\sim (vK_0\frac{2}{n_{\tau}}e)^k \frac{1}{\sqrt{2\pi k}}
    \end{align*}
    which converges exponentially fast towards $0$ for $v=\ell_{\tau}$ where we have $\ell_{\tau}:=\frac{1}{2}\frac{n_{\tau}}{2e}2^{-\frac{2}{n_{\tau}}}L^{-\frac{2}{n_{\tau}}}C_0^{-\frac{2}{n_{\tau}}}C_{n_{\tau}}^{-\frac{2}{n_{\tau}}} C_1^{-\frac{1}{n_{\tau}}}>0$. \textcolor{black}{We emphasize that $\ell_{\tau}$ still depends of $\lambda$ since our definition of $L$ involves $\lambda$.}\\
    
    \textcolor{black}{We are close to the end of the proof, but before concluding we have to talk about homogeneity in $\sigma$. We first underscore that the constant in front of $\xi$ in the definition of $\RSI$ is not $1$ but $\sigma$, and therefore all occurrences of $\xi$ in equations are replaced by $\xi^{\sigma}=\sigma \xi$ and the iterated stochastic integrals are against element of the form $\xi^{\sigma}(\mathrm{d}z_1)\cdots\xi^{\sigma}(\mathrm{d}z_k)=\sigma^k \xi(\mathrm{d}z_1)\cdots \xi(\mathrm{d}z_k)$. We now recall that $\tau$ is a process in $\mathcal{H}_{\le n_{\tau}}$, so we can write} \begin{equation}\textcolor{black}{\tau=\sum_{\ell=0}^{n_{\tau}} T_\ell,} \label{hom}\end{equation}
    \textcolor{black}{with $T_\ell$ a process in $\mathcal{H}_\ell$. Since any element of $\mathcal{H}_{\ell}$ could be written as an iterated stochastic integral of order $\ell$ (see Appendix), the previous paragraph implies that $T_\ell=\sigma^\ell \tilde{T}_\ell$ with $\tilde{T}_\ell$ independent of $\sigma$. We wrote the computations that preceded with $\tau$ in order not to overload the presentation, but the reader can check that they trivially work for all $T_\ell$.}\\
    
    We can now conclude using Markov's inequality. We have that $\Pi_k \tau=\sigma^{k}\tilde{T}_{k}$ so
    \begin{align*}
        \mathbb{P}([\Pi_k\tau]_{\beta,\gamma}>h^{k})&=\mathbb{P}([\sigma^{k}\tilde{T}_{k}(t)]_{\beta,\gamma}>h^{k}) \\
        &\le\mathbb{P}([\tilde{T}_{k}(t)]_{\beta,\gamma}>\frac{h^{k}}{\sigma^{k}}) \\
        &\le\mathbb{P}\left(\exp(\ell_{\tau}[\tilde{T}_{k}(t)]_{\beta,\gamma}^{\frac{2}{k}})>\exp(\ell_{\tau}\frac{h^{2}}{\sigma^{2}})\right) \\
        &\le \exp(-\ell_{\tau} \frac{h^2}{\sigma^2})\mathbb{E}(\exp(\ell_{\tau}[\tilde{T}_{k}(t)]_{\beta,\gamma}^{\frac{2}{k}}))
    \end{align*}
We then take $m_{\tau}=\mathbb{E}(\exp(\ell_{\tau}[\tilde{T}_{k}(t)]_{\beta,\gamma}^{\frac{2}{k}}))$ which is finite and independent of $T$ and we are done.

\end{proof}
\textcolor{black}{Let us now state a direct Corollary that we will also invoke in the proof of Theorem \ref{thp}}
\begin{corollary} \label{chc}
    Let $\tau\in \mathcal{T}$, and $\beta<|\tau|-\lambda$ for $\lambda\in (0,1)$. Then there exist $\ell_{\tau}':=\ell_{\tau}'(\lambda,\textcolor{black}{\beta,a_{+},a_{-}})$ and $m_{\tau}':=m_{\tau}'(\lambda,\textcolor{black}{\beta,a_{+},a_{-}})$ independent of $T$ such that for $k\le n_{\tau}$ we have
    $$\mathbb{P}\left([\Pi_k\tau]_{\beta,\frac{\lambda}{2}}>h^{k}\right)\le m_{\tau}' \exp\left(-\ell_{\tau}'\frac{h^2}{\sigma^2}\right).$$
\end{corollary}
\begin{proof}
    We just have to use Proposition \ref{ch} for $\lambda$ and $\lambda'\in (\lambda,|\tau|-\beta)$ where the upper bound gives us that $\beta<|\tau|-\lambda'$. There exists $\gamma_0(T,\lambda)<\frac{\lambda}{2}$ and $\tilde{\gamma}_0(T,\lambda')<\frac{\lambda'}{2}$ such that we have the wanted property \textcolor{black}{for all} $\gamma_1\in (\gamma_0,\frac{\lambda}{2})$ and $\gamma_2\in(\tilde{\gamma}_0,\frac{\lambda'}{2})$, so we take $\gamma_1<\frac{\lambda}{2}<\gamma_2$ and then use that $$[\Pi_k \tau]_{\beta,\frac{\lambda}{2}}\le [\Pi_k \tau]_{\beta,\gamma_1,}+[\Pi_k \tau]_{\beta,\gamma_2}$$
    to conclude.
\end{proof}
\textcolor{black}{We can now prove} Theorem \ref{cs}, \textcolor{black}{that} is a direct consequence of Corollary \ref{chc}. Indeed, for $\alpha<|\tau|-\lambda$ where we can choose $\lambda$ arbitrarily small, using that $\tau(0)=0$ we have
\begin{align*}
    \mathbb{P}\left(\sup_{t\in [0,T]} \|\Pi_k\tau(t)\|_{\mathcal{C}^{\alpha}}>h^{k}\right)&= \mathbb{P}(\sup_{t\in [0,T]}\|\Pi_k\tau(t)-\Pi_k\tau(0)\|_{\mathcal{C}^{\alpha}}>h^k) \\
    &\le \mathbb{P}(T^{\frac{\lambda}{2}} [\Pi_k \tau]_{\alpha,\frac{\lambda}{2}}>h^k)\\
    &\le m_{\tau}'\exp(-\frac{\ell_{\tau}'}{(T^{\frac{\lambda}{2}})^{\frac{2}{k}}}\frac{h^2}{\sigma^2})\\
    &\le m_{\tau}'\exp(-\frac{\ell_{\tau}'}{(T^{\lambda})^{\frac{1}{k}}}\frac{h^2}{\sigma^2})
\end{align*}
and the proof is complete for $d_{\tau}=m_{\tau}'$ and $c_{\tau}=\ell_{\tau}'$

\subsection{Proof of the main theorems}
\textcolor{black}{In this subsection, we will first give an explicit expression of the functions $F$ and $G$ introduced in Lemma \ref{defvw} and explain why \eqref{e:eqvw} admits a regular solution. Then, we will prove several technical Lemmas that we will combine with Theorem \ref{cs} to prove Theorem \ref{thp}. Finally, we will deduce Corollary \ref{corp} from Theorem \ref{thp}.}
\subsubsection{Proof of Lemma \ref{defvw} \label{ce}}
We start by doing the computations mentioned \textcolor{black}{at} the end of Section \ref{dth}. They are essentially the same as the ones found in \cite{MW}, but with extra terms due to $f_2$. We need to renormalize a second time \textcolor{black}{the equation of $\phi$}, and we therefore \textcolor{black}{consider} $\theta:=\lim_{n\to+\infty} \theta_n$ where $\theta_n$ is the unique solution with initial condition $0$ of the equation
\begin{align*}
(\partial_t-\Delta -a(t))\theta_n&=-\theta_n^3-3\RSI_n \theta_n^2-3\RSV_n\theta_n -\RSW_n \\ & \ \ \ \ \  +f_2(t)\theta_n^2+2f_2(t) \theta_n \RSI+f_2(t)\RSV_n\textcolor{red}{-6\tilde{c}_n(t)(3\RSI_n+3\theta_n)} .
\end{align*}
We then take $u_n=\theta_n+\RSIW_n$ and decompose the product $\RSV_n(u_n-\RSIW_n)$ into its paraproducts and resonant product parts
\begin{align*}\textcolor{black}{\RSV_n(u_n-\RSIW_n)}&=\textcolor{black}{\RSV_n\pl (u_n-\RSIW_n)+\RSV_n\pe (u_n-\RSIW_n)} \\ & \ \ \ \ \ \ \ \ \ \textcolor{black}{+\RSV_n\pg (u_n-\RSIW_n)}\end{align*} Since $(\partial_t-\Delta -a(t))\RSIW_n=\RSW_n$ we have
\begin{align*}
&(\partial_t-\Delta -a(t))u_n \\ &=-(u_n-\RSIW_n)^3-3(u_n-\RSIW_n)^2 \RSI_n -3(u_n-\RSIW_n)\RSV_n\\ & \ \ \ \ \ \ +f_2(t)(u_n-\RSIW_n)^2+2f_2(t) (u_n-\RSIW_n) \RSI+f_2(t)\RSV_n \\
& \ \ \ \ \ \ \ \ \ \ \ \ -6\tilde{c}_n(t)(3\RSI_n+3(u-\RSIW_n)) \\
&=-u_n^3-3(u_n\pe \RSV_n+6\tilde{c}_n(t)(u_n-\RSIW_n))+3(\RSIW_n\pe\RSV_n-6\tilde{c}_n(t)\RSI_n)\\ & \ \ \ \ \ \ -3(u_n-\RSIW_n)\pl \RSV_n-3(u_n-\RSIW_n)\pg \RSV_n+Q^n(u_n) \\
&=-u_n^3-3(u_n\pe \RSV_n+6\tilde{c}_n(t)(u_n-\RSIW_n))+3\RSWW_n\\ & \ \ \ \ \ \ -3(u_n-\RSIW_n)\pl \RSV_n-3(u_n-\RSIW_n)\pg \RSV_n
+Q^n(u_n)\end{align*}
where we have $Q^n(u_n)=q_2^n(t)u_n^2+q_1^n(t)u_n+q_0^n(t)$ with:
\begin{align*}
    &q_0^n(t)=(\RSIW_n)^3-3\RSI_n (\RSIW_n)^2+f_2(t) (\RSIW_n)^2-2f_2(t)\RSIW_n \RSI_n+f_2(t)\RSV_n  \\
    &q_1^n(t)=-3(\RSIW_n)^2+6\RSI_n \RSIW_n-2f_2(t)\RSIW_n+2f_2(t)\RSI_n \\
    &q_2^n=-3\RSI_n+3\RSIW_n+f_2(t)
\end{align*}
\textcolor{black}{As we said in Section \ref{dth}, the term $+6\tilde{c}_n(t)(u_n-\RSIW_n)$ is diverging in $n$, and it is not clear how it compensates with $u_n\pe \RSV_n$. To end up with a limit equation, we have to cut $u_n$ into three parts and apply to them a non-trivial procedure that will generates new $\RSWV_n$ terms to be put in $q_0$ and $q_1$. We will not enter into the technical details, but} following exactly the same method as in \cite{MW} (the term with $m$ is here replaced with $a(t)$ and put in the left-hand side instead of the right-hand side) and putting the new $+f_2(t)\RSV_n$ term of $q_0(t)$ \textcolor{black}{into} the equation of $v_n$, we can write $u_n=v_n+w_n+3I(\RSWW_n)$ such that $(v,w)\textcolor{black}{:=\lim_{n\to +\infty}(v_n,w_n)}$ \textcolor{black}{should satisfy} \eqref{e:eqvw}

\begin{equation}\tag{\ref{e:eqvw}}
\left\{
\begin{array}{lll}
(\partial_t - \Delta-a(t)) v & = & F(v+w),\\
(\partial_t - \Delta-a(t)) w & = & G(v,w),
\end{array}
\right.
\end{equation}
where the explicit expression of $F$ and $G$ are
\begin{equation} \label{expFG}
\begin{array}{ll}F(v+w) & := -3(v+w- \RSIW ) \pl \RSV+f_2(t)\RSV, \\
G(v,w) & := -(v+w)^3  - 3 \mathsf{com}(v,w) \\
& \qquad \qquad -3w \pe \RSV - 3(v+w-\RSIW) \pg\RSV + P(v+w) \end{array}
\end{equation} 
where $P(X)=d_2(t) X^2+d_1(t)X+d_0(t)$ is a random polynomial of coefficients:
\begin{align*}
    &d_0(t)= (\RSIW)^3-9 \RSIW \RSWV +f_2(t) (\RSIW)^2-2f_2(t)[\RSVW+\RSI\pne \RSIW] \\ & \ \ \ \ \ \ \ \ \ -3\left[\RSI \pne (\RSIW)^2+\RSI \pe [\RSIW \pe \RSIW]+2 \RSIW \RSVW +2[\pl,\pe](\RSIW,\RSIW,\RSI)\right] \\
    &d_1(t)=6\left[\RSIW \pne \RSI+\RSVW \right]-3(\RSIW)^2+9\RSWV -2f_2(t)\RSIW+2f_2(t)\RSI\\
    &d_2(t)=-3\RSI+3\RSIW+f_2(t)
\end{align*}
and $\mathsf{com}(v,w)=\text{com}_1(v,w)\pe \RSV+\text{com}_2(v+w)$ where:
$$\left \{ \begin{array}{ll}
     \text{com}_1(v,w)(t)&=v(t)+[3(v+w-\RSIW)\textcolor{black}{-f_2}](t)\textcolor{black}{\pl \RSY(t)}  \\
     \text{com}_2(v+w)&= [\pl,\pe](-3(v+w-\RSIW),\RSY,\RSV)
\end{array} \right.$$

\begin{remark} The removal of the term $+3I(\RSWW)$ from the definition of $w$ found in \cite{MW} implies the disappearance of the term $3\RSWW$ from the definition of $d_0$.\end{remark}
\textcolor{black}{However we did not prove yet that \eqref{e:eqvw} is well-posed. It is precisely what \cite{MW}[Theorem 2.1] says, while in addition giving us that for $(v_0,w_0)=(0,0)\in \mathcal{C}^{1-2\varepsilon}\times \mathcal{C}^{\frac{3}{2}-2\varepsilon}$ we have that $v$ belongs to $\mathcal{C}([0,T],\mathcal{C}^{1-2\varepsilon})\cap \mathcal{C}^{\frac{1}{8}+\delta}([0,T],L^{\infty})$ and $w$ to $\mathcal{C}([0,T],\mathcal{C}^{\frac{3}{2}-2\varepsilon})\cap \mathcal{C}^{\frac{1}{8}+\delta}([0,T],L^{\infty})$ for $\delta$ small enough.} 
\begin{remark}
    \textcolor{black}{In fact, the statement made in \cite{MW}[Theorem 2.1] is slightly weaker than that, giving us only a spatial regularity of $\frac{1}{2}+2\varepsilon$ for $v$ and $1+2\varepsilon$ for $w$. However, the regularity considered in \cite{MW}[Theorem 2.1] is sub-optimal, and thanks to \cite{MW}[Remark 2.5] we have the wanted result. The paragraphs following this remark also assure us that for $(v_0,w_0)=(0,0)$, the solution of \eqref{e:eqvw} is global.}
\end{remark}

\textcolor{black}{Before proving the main theorem we will state and prove a technical Lemma that will reveal useful latter.}
\begin{lemma}\label{cont}
    \textcolor{black}{The application $t\in[0,T]\mapsto \|(v,w)\|_{\Xi,t}$ is continuous.}  
\end{lemma}
\begin{proof}
    \textcolor{black}{We know that $(v,w)\in \mathcal{C}([0,T],\mathcal{C}^{1-2\varepsilon})\times\mathcal{C}([0,T],\mathcal{C}^{\frac{3}{2}-2\varepsilon})$ so $t\mapsto \max(\sup_{s\in[0,t]}\|v(s)\|_{\mathcal{C}^{1-2\varepsilon}},\sup_{s\in[0,t]}\|w(s)\|_{\mathcal{C}^{\frac{3}{2}-2\varepsilon}})$ is by definition continuous.}\\

    \textcolor{black}{To deal with the $[v_{|[0,t]}]$ and $[w_{|[0,t]}]$ terms we have to use the fact that $(v,w)$ is not only in $\mathcal{C}^{\frac{1}{8}}([0,T],L^{\infty})\times \mathcal{C}^{\frac{1}{8}}([0,T],L^{\infty})$ but in $\mathcal{C}^{\frac{1}{8}+\delta}([0,T],L^{\infty})\times \mathcal{C}^{\frac{1}{8}+\delta}([0,T],L^{\infty})$.}\\

    \textcolor{black}{The triangular inequality implies $[v_{|[0,t+\eta]}]_{0,\frac{1}{8}}\le [v_{|[0,t]}]_{0,\frac{1}{8}}+[v_{|[t,t+\eta]}]_{0,\frac{1}{8}}$ and $[v_{|[0,t]}]_{0,\frac{1}{8}}\le [v_{|[0,t-\eta]}]_{0,\frac{1}{8}}+[v_{|[t-\eta,t]}]_{0,\frac{1}{8}}$.  The application $t\mapsto [v_{|[0,t]}]$ being increasing, we therefore just have to prove that $[v_{|[t-\eta,t]}]_{0,\frac{1}{8}}$ and $[v_{|[t,t+\eta]}]_{0,\frac{1}{8}}$ go to $0$ when $\eta$ goes to $0$ to prove that it is also left and right-continuous (i.e. continuous). Since $v\in \mathcal{C}^{\frac{1}{8}+\delta}([0,T],L^{\infty})$, we have that $[v_{|[t,t+\eta]}]_{0,\frac{1}{8}}\le \eta^{\delta}[v_{|[t,t+\eta]}]_{0,\frac{1}{8}+\delta}\le \eta^{\delta}[v]_{0,\frac{1}{8}+\delta}\underset{\delta\to 0}{\to}0$. This proves the result for $v$, and since the proof also works for $w$, we are done.}
\end{proof}

\subsubsection{Proof of Theorem \ref{thp} \label{pthp}}

We write $C$ for a universal constant independent of $T$ and $\sigma$ to simplify the notations. In what follows, if we do not explicitly write the temporal variable, it means that the result is true for all $s\in [0,T]$ and that the constants involved in the computations are uniform in time. Since proving Theorem \ref{thp} for a specific $\varepsilon>0$ immediately implies the result for all $\varepsilon'\ge \varepsilon$, we assume from now on that $\varepsilon\in (0,\frac{1}{16})$. \\

\textcolor{black}{To control $\|(v,w)\|_{\Xi,T}$, we need to control four terms: $\sup_{s\in[0,T]}\|v(s)\|_{\mathcal{C}^{1-2\varepsilon}}$, $\sup_{s\in[0,T]}\|w(s)\|_{\mathcal{C}^{\frac{3}{2}-2\varepsilon}}, [v]_{0,\frac{1}{8}}$ and $[w]_{0,\frac{1}{8}}$. From \eqref{e:eqvw} we can already deduce bounds on them involving $F$ and $G$.}
\begin{lemma} \label{bornefg}
    \textcolor{black}{Let $t_0\in [0,T]$. We have the following four bounds:}
    \begin{align*}
        \sup_{r\in[0,t_0]} \|v(r)\|_{\mathcal{C}^{1-2\varepsilon}}&\le C \sup_{r\in [0,t_0]}\|F(v+w)(r)\|_{\mathcal{C}^{-1-\varepsilon}} \numberthis \label{majv}\\
   \sup_{r\in[0,t_0]}\|w(r)\|_{\mathcal{C}^{\frac{3}{2}-2\varepsilon}}&\le C \sup_{r\in [0,t_0]} \|G(v,w)(r)\|_{\mathcal{C}^{-\frac{1}{2}-\varepsilon}}\numberthis\label{majw}\\
    [v_{|[0,t_0]}]_{0,\frac{1}{8}}&\le C \sup_{r\in [0,t_0]} \|F(v+w)(r)\|_{\mathcal{C}^{-1-\varepsilon}}, \numberthis \label{lipv} \\ [w_{|[0,t_0]}]_{0,\frac{1}{8}}&\le C \sup_{r\in [0,t_0]} \|G(v,w)(r)\|_{\mathcal{C}^{-\frac{1}{2}-\varepsilon}}. \numberthis \label{lipw} \end{align*}

\end{lemma}
\textcolor{black}{The proof is left to Section \ref{auxl}.}\\

Lemma \ref{bornefg} implies that to control $\|(v,w)\|_{\Xi,T}$, we only need to control $F(v+w)$ and $G(v,w)$. Therefore, we will look for precise bound on these two terms in the following pages. Since these bounds will involve $\|v\|$, $\|w\|$, $[v]$ and $[w]$, that is to say the quantities we want to control in the first place, they may appear at first glance unusable. We will explain at the end of the proof how we can use them nonetheless. \\

\textcolor{black}{Starting with $F$, we have by} Proposition \ref{ip}
\begin{align*}
    &\|F(v+w)\|_{\mathcal{C}^{-1-\varepsilon}} \\ &\le 3\|(v+w- \RSIW ) \pl \RSV\|_{\mathcal{C}^{-1-\varepsilon}}+\|f_2\|_{\infty}\|\RSV\|_{\mathcal{C}^{-1-\varepsilon}}\\
    &\le C \|v+w-\RSIW\|_{\mathcal{C}^{\frac{1}{2}-\varepsilon}} \times \|\RSV\|_{\mathcal{C}^{-1-\varepsilon}}+\|f_2\|_{\infty}\|\RSV\|_{\mathcal{C}^{-1-\varepsilon}} \\
    &\le C (\|v\|_{\mathcal{C}^{1-2\varepsilon}}+\|w\|_{\mathcal{C}^{\frac{3}{2}-2\varepsilon}}+\|\RSIW\|_{\mathcal{C}^{\frac{1}{2}-\varepsilon}}+\|f_2\|_{\infty}) \|\RSV\|_{\mathcal{C}^{-1-\varepsilon}}. \numberthis \label{majF}
\end{align*}

\textcolor{black}{Regarding $G$, since its expression} is more complex than the one of $F$, we will study its five terms one after the other. For the cubic term $(v+w)^3$ we have
\begin{align*}
    \|-(v+w)^3\|_{\mathcal{C}^{-\frac{1}{2}-\varepsilon}}&\le \|(v+w)^3\|_{\mathcal{C}^{1-2\varepsilon}} \\
    &\le C \|v+w\|_{\mathcal{C}^{1-2\varepsilon}}^3 \\
    &\le C(\|v\|_{\mathcal{C}^{1-2\varepsilon}}^3+\|w\|_{\mathcal{C}^{\frac{3}{2}-2\varepsilon}}^3). \numberthis \label{majG1}
\end{align*}
If we consider the resonant product of $\RSV$ and $w$ we have, thanks to Proposition, \ref{ip}
\begin{align*}
    \|3w\pe \RSV\|_{\mathcal{C}^{-\frac{1}{2}-\varepsilon}}&\le \|3w\pe \RSV\|_{\mathcal{C}^{\frac{1}{2}-3\varepsilon}} \\
    &\le C \|w\|_{\mathcal{C}^{\frac{3}{2}-2\varepsilon}} \|\RSV\|_{\mathcal{C}^{-1-\varepsilon}}. \numberthis \label{majG2}
\end{align*}
For the fourth term we have likewise
\begin{align*}
    &\|3(v+w-\RSIW) \pg\RSV\|_{\mathcal{C}^{-\frac{1}{2}-\varepsilon}} \\ &\le C \|v+w-\RSIW\|_{\mathcal{C}^{\frac{1}{2}-\frac{\varepsilon}{2}}} \|\RSV\|_{\mathcal{C}^{-1-\frac{\varepsilon}{2}}} \\
    &\le C(\|\RSIW\|_{\mathcal{C}^{\frac{1}{2}-\frac{\varepsilon}{2}}}+\|v\|_{\mathcal{C}^{1-2\varepsilon}}+\|w\|_{\mathcal{C}^{\frac{3}{2}-2\varepsilon}})\|\RSV\|_{\mathcal{C}^{-1-\frac{\varepsilon}{2}}}. \numberthis \label{majG3}
\end{align*}
With regards to the polynomial $P$ we have
\begin{align*}
    \|P(v+w)\|_{\mathcal{C}^{-\frac{1}{2}-\varepsilon}}&\le C\left[\|d_0\|_{\mathcal{C}^{-\frac{1}{2}-\varepsilon}}+\|d_1\|_{\mathcal{C}^{-\frac{1}{2}-\varepsilon}}\|v+w\|_{\mathcal{C}^{1-2\varepsilon}}\right. \\ & \ \ \ \ \ \  \ \ \ \ \ \left. +\|d_2\|_{\mathcal{C}^{-\frac{1}{2}-\varepsilon}}\|(v+w)^2\|_{\mathcal{C}^{1-2\varepsilon}}\right] \\
    &\le C\left[\|d_0\|_{\mathcal{C}^{-\frac{1}{2}-\varepsilon}}+\|d_1\|_{\mathcal{C}^{-\frac{1}{2}-\varepsilon}}(\|v\|_{\mathcal{C}^{1-2\varepsilon}}+\|w\|_{\mathcal{C}^{\frac{3}{2}-2\varepsilon}}) \right. \\ & \ \ \ \ \ \ \ \ \ \ \ \left. +\|d_2\|_{\mathcal{C}^{-\frac{1}{2}-\varepsilon}}(\|v\|_{\mathcal{C}^{1-2\varepsilon}}^2+\|w\|_{\mathcal{C}^{\frac{3}{2}-2\varepsilon}}^2)\right] \numberthis \label{majG4}
\end{align*}
The term involving $\text{com}$ is more complex. 
We have on the one hand, thanks to Proposition \ref{com}, that
\begin{align*}
    &\|\text{com}_2(v+w)\|_{\mathcal{C}^{-\frac{1}{2}-\varepsilon}} \\ &\le \|\text{com}_2(v+w)\|_{\mathcal{C}^{\frac{1}{2}-\varepsilon}} \\
    &\le C (\|v+w-\RSIW\|_{\mathcal{C}^{\frac{1}{2}-\frac{1}{3}\varepsilon}} \|\RSY\|_{\mathcal{C}^{1-\frac{1}{3}\varepsilon}} \|\RSV\|_{\mathcal{C}^{-1-\frac{1}{3}\varepsilon}}) \\
    &\le C \left( \|v\|_{\mathcal{C}^{1-2\varepsilon}}+\|w\|_{\mathcal{C}^{\frac{3}{2}-2\varepsilon}}+\|\RSIW\|_{\mathcal{C}^{\frac{1}{2}-\frac{1}{3}\varepsilon}}\right) \|\RSY\|_{\mathcal{C}^{1-\frac{1}{3}\varepsilon}} \|\RSV\|_{\mathcal{C}^{-1-\frac{1}{3}\varepsilon}} \numberthis \label{majG5}
\end{align*}
For $\text{com}_1$ on the other hand, integrating \eqref{e:eqvw} \textcolor{black}{and then using that $\RSV=1\pl \RSV+1\pge \RSV$} we have that \begin{align*}v(t)&=-3\int_0^t e^{\alpha(t,s)} e^{(t-s)\Delta} [(v+w-\RSIW)\pl \RSV](s)\mathrm{d}s \\ & \ \ \ \ \ \ \ \ \ \textcolor{black}{+\int_0^t e^{\alpha(t,s)} e^{(t-s)\Delta} f_2(s)\RSV(s)\mathrm{d}s} \\
&=-\int_0^t e^{\alpha(t,s)} e^{(t-s)\Delta} [3(v+w-\RSIW)\textcolor{black}{-f_2})\pl \RSV](s)\mathrm{d}s \\ & \ \ \ \ \ \ \ \ \ \textcolor{black}{+\int_0^t e^{\alpha(t,s)} e^{(t-s)\Delta} f_2(s)(1\pge\RSV(s))\mathrm{d}s}
\end{align*} and therefore \begin{align*}\text{com}_1(v,w)(t)&=-\int_0^t e^{\alpha(t,s)}e^{(t-s)\Delta} [3(v+w-\RSIW)\textcolor{black}{-f_2}]\pl \RSV](s)\mathrm{d}s \\ & \ \ \ \ \ \ \ \ +[3 (v+w-\RSIW)\textcolor{black}{-f_2}](t)\pl \RSY(t) \\
& \ \ \ \ \ \ \ \ \ \ \ \ \ \ \ \  \textcolor{black}{+\int_0^t e^{\alpha(t,s)} e^{(t-s)\Delta} f_2(s)(1\pge\RSV(s))\mathrm{d}s}.\end{align*} The computations associated with $\text{com}_1$ are the most subtle \textcolor{black}{part}. \textcolor{black}{We} recall from \cite{MW} that the motivation behind the definition of $\text{com}_1(v,w)$ is that we expect it to be a bit more regular than $v$ so that $\text{com}_1(v,w)\pe \RSV$ is well defined (while $v\pe \RSV$ is not). \\

\textcolor{black}{Indeed, while $v$ and $[3(v+w-\RSIW)-f_2]\pl \RSY$ are of regularity $\mathcal{C}^{-1-\varepsilon}$,} we \textcolor{black}{will show that we can} rewrite $\text{com}_1(v,w)(t)=A_t+B_t\textcolor{black}{+C_t}$ \textcolor{black}{where $C_t=\int_0^t e^{\alpha(t,s)} e^{(t-s)\Delta} f_2(s)(1\pge\RSV(s))\mathrm{d}s$ is in $\mathcal{C}^{\infty}$ since $1$ is in $\mathcal{C}^{\infty}$ (see Proposition \ref{ip}), and both $A$ and $B$ are of regularity $\mathcal{C}^{1+2\varepsilon}$.} \textcolor{black}{More precisely, we have}
\begin{align*}
    \textcolor{black}{\text{com}_1(v,w)(t)}&=-\int_0^t e^{\alpha(t,s)}  e^{(t-s)\Delta} [3(v+w-\RSIW)\textcolor{black}{-f_2})\pl \RSV](s)\mathrm{d}s \\
    &  \ \ \ \ \ \ \ +[3 (v+w-\RSIW)\textcolor{black}{+f_2}](t)\pl \RSY(t) \\
    & \ \ \ \ \ \ \  \ \ \ \ \ \ \  \textcolor{black}{+\int_0^t e^{\alpha(t,s)} e^{(t-s)\Delta} f_2(s)(1\pge\RSV(s))\mathrm{d}s} \\
     &=-\int_0^t  e^{\alpha(t,s)} [e^{(t-s)\Delta},\pl]([3(v+w-\RSIW)\textcolor{black}{-f_2}],\RSV)(s)\mathrm{d}s \\
    & \ \ \ \ \ \ \ -\int_0^t e^{\alpha(t,s)} [[3(v+w-\RSIW)\textcolor{black}{-f_2}]\pl   e^{(t-s)\Delta}\RSV](s)\mathrm{d}s \\
    & \ \ \ \ \ \ \ \ \ \ \ \ \ \ + [3(v+w-\RSIW)\textcolor{black}{-f_2}](t)\pl \RSY(t) \\
    & \ \ \ \ \ \ \ \ \ \ \ \ \ \ \ \ \ \ \ \ \ \textcolor{black}{+\int_0^t e^{\alpha(t,s)} e^{(t-s)\Delta} f_2(s)(1\pge\RSV(s))\mathrm{d}s}
\end{align*}
where $[e^{(t-s)\Delta},\pl](f,g)=e^{(t-s)\Delta}(f\pl g)-f\pl(e^{(t-s)\Delta} g)$. From now on, we denote by $A_t$ the first term on the \textcolor{black}{right}-hand side \textcolor{black}{of the equation above, $B_t$ the sum of the second and third ones, and $C_t$ the fourth one}. \textcolor{black}{Recall} that $\RSY(t)=\int_0^t e^{\alpha(t,s)} e^{(t-s)\Delta}\RSV(s)\mathrm{d}s$ \textcolor{black}{and let} $\delta_{s,t} f=f(s)-f(t)$. \textcolor{black}{We have that}
\begin{align*}
    \textcolor{black}{B_t:=}&-\int_0^t  [[3(v+w-\RSIW)\textcolor{black}{-f_2}]\pl  e^{\alpha(t,s)} e^{(t-s)\Delta}\RSV](s)\mathrm{d}s\\ & \ \ \ \ + [3(v+w-\RSIW)\textcolor{black}{-f_2}]\pl \RSY(t) \\
    &=-\int_0^t [[3(v+w-\RSIW)\textcolor{black}{-f_2}]\pl  e^{\alpha(t,s)}  e^{(t-s)\Delta} \RSV](s)\mathrm{d}s \\ & \ \ \ \ \ +[3(v+w-\RSIW)\textcolor{black}{-f_2}](t)\pl \left[\int_0^t  e^{\alpha(t,s)} e^{(t-s)\Delta}\RSV(s)\mathrm{d}s \right] \\
    &=-\int_0^t  [[\delta_{s,t}[3(v+w-\RSIW)\textcolor{black}{-f_2}]]\pl  e^{\alpha(t,s)} e^{(t-s)\Delta}\RSV(s)]\mathrm{d}s.   
\end{align*}
\textcolor{black}{Eventually}, using Proposition \ref{come} to bound the norm of $A_t$ and Proposition \ref{ip} combined with Proposition \ref{es} to bound the one\textcolor{black}{s} of $B_t$ \textcolor{black}{and $C_t$}, we get
\begin{align*}
    &\|\text{com}_1(v,w)(t)\|_{\mathcal{C}^{1+2\varepsilon}} \\ &\le \|A_t\|_{\mathcal{C}^{1+2\varepsilon}}+\|B_t\|_{\mathcal{C}^{1+2\varepsilon}}+\textcolor{black}{\|C_t\|_{\mathcal{C}^{1+2\varepsilon}}} \\ 
    &\le C \int_0^t e^{-a_{+}(t-s)}(t-s)^{\frac{\frac{1}{2}-\varepsilon-1-\varepsilon-(1+2\varepsilon)}{2}}\|[3(v+w-\RSIW)\textcolor{black}{-f_2}](s)\|_{\mathcal{C}^{\frac{1}{2}-\varepsilon}} \\ & \ \ \ \ \ \ \ \ \ \ \ \ \ \ \ \ \ \ \ \ \ \ \ \ \ \ \ \ \ \ \ \ \ \ \ \ \ \ \ \ \ \ \ \ \ \ \ \ \ \ \ \ \ \ \ \ \ \times \|\RSV(s)\|_{-1-\varepsilon}\mathrm{d}s \\
   & \ \ \ \ \ \ \ \ + C \int_0^t \|\delta_{s,t}[3(v+w-\RSIW)\textcolor{black}{-f_2}]\|_{L^{\infty}} \| e^{\alpha(t,s)} e^{(t-s)\Delta} \RSV(s)\|_{\mathcal{C}^{1+2\varepsilon}}\mathrm{d}s \\
   & \ \ \ \ \ \ \ \ \ \ \ \ \ \ \ \ \textcolor{black}{+ C \int_0^t \| e^{\alpha(t,s)} e^{(t-s)\Delta} f_2(s)(1\pge\RSV)(s)\|_{\mathcal{C}^{1+2\varepsilon}}\mathrm{d}s} \\
   &\le C \sup_{s\in[0,t]}\|[3(v+w-\RSIW)\textcolor{black}{-f_2}](s)\|_{\mathcal{C}^{\frac{1}{2}-\varepsilon}}\|\RSV(s)\|_{\mathcal{C}^{-1-\varepsilon}} \\
   & \ \ \ \ \ \ \ \ + C \int_0^t \|\delta_{s,t}[3(v+w-\RSIW)\textcolor{black}{-f_2}]\|_{L^{\infty}} \, e^{-a_{+}(t-s)}  (t-s)^{\frac{-1-\varepsilon-1-2\varepsilon}{2}} \\ & \ \ \ \ \ \ \ \ \ \ \ \ \ \ \ \ \ \ \ \ \ \ \ \ \ \times\|\RSV(s)\|_{\mathcal{C}^{-1-\varepsilon}}\mathrm{d}s \\
   & \ \ \ \ \ \ \ \ \ \ \ \ \ \ \ \ \textcolor{black}{+C\sup_{s\in[0,t]}|f_2(s)|\cdot\|\RSV(s)\|_{\mathcal{C}^{-1-\varepsilon}}}
\end{align*}
We then use the temporal regularity of $v,w$ and $\RSIW$ to argue that \begin{align*}&\|\delta_{s,t}[3(v+w-\RSIW)\textcolor{black}{-f_2}]\|_{L^{\infty}} \\ &\le \textcolor{black}{3[(v+w-\RSIW)_{|[0,t]}]_{0,\frac{1}{8}}} (t-s)^{\frac{1}{8}}+\textcolor{black}{\|f_2'\|_{\infty}(t-s)} \\ &\le\textcolor{black}{3\left([v_{|[0,t]}]_{0,\frac{1}{8}}+[w_{|[0,t]}]_{0,\frac{1}{8}}+[\RSIW_{|[0,t]}]_{0,\frac{1}{8}}\right)} (t-s)^{\frac{1}{8}}+\textcolor{black}{\|f_2'\|_{\infty}(t-s)},\end{align*} so the integral above converges since $\varepsilon<\frac{1}{16}$, and we have
\begin{align*}
    &\|\text{com}_1(v,w)(t)\|_{\mathcal{C}^{1+2\varepsilon}} \\ &\le C( \sup_{s\in [0,t]} [\|v(s)\|_{\mathcal{C}^{1-2\varepsilon}}+\|w(s)\|_{\mathcal{C}^{\frac{3}{2}-2\varepsilon}}+\|\RSIW(s)\|_{\mathcal{C}^{\frac{1}{2}-\varepsilon}}\textcolor{black}{+\|f_2\|_{\infty}}]\|\RSV(s)\|_{\mathcal{C}^{-1-\varepsilon}} \\ & \ \ \ \ \ \ \ \ +\textcolor{black}{\left([v_{|[0,t]}]_{0,\frac{1}{8}}+[w_{|[0,t]}]_{0,\frac{1}{8}}+[\RSIW_{|[0,t]}]_{0,\frac{1}{8}}+\|f_2'\|_{\infty}\right)} \sup_{s\in [0,t]} \|\RSV(s)\|_{\mathcal{C}^{-1-\varepsilon}})
\end{align*}

This quantity being finite, $\text{com}_1(v,w)(t)$ belongs to $\mathcal{C}^{1+2\varepsilon}$ and the product $\text{com}_1(v,w)(t)\pe \RSV(t)$ is therefore well-defined. Besides, Proposition \ref{ip} gives us that
\begin{align*}&\|\text{com}_1(v,w)\pe \RSV (t)\|_{\mathcal{C}^{\varepsilon}} \\ &\le C \|\text{com}_1(v,w)(t)\|_{\mathcal{C}^{1+2\varepsilon}} \|\RSV(t)\|_{\mathcal{C}^{-1-\varepsilon}} \\
&\le C \left\{ \sup_{s\in [0,t]} \left[\|v(s)\|_{\mathcal{C}^{1-2\varepsilon}}+\|w(s)\|_{\mathcal{C}^{\frac{3}{2}-2\varepsilon}}+\|\RSIW(s)\|_{\mathcal{C}^{\frac{1}{2}-\varepsilon}}\textcolor{black}{+\|f_2\|_{\infty}}\right] \right. \\
& \ \ \ \ \ \  \left. +\textcolor{black}{\left([v_{|[0,t]}]_{0,\frac{1}{8}}+[w_{|[0,t]}]_{0,\frac{1}{8}}+[\RSIW_{|[0,t]}]_{0,\frac{1}{8}}+\|f_2'\|_{\infty}\right)} \right\}\sup_{s\in [0,t]}  \|\RSV(s)\|_{\mathcal{C}^{-1-\varepsilon}}^2 . \numberthis \label{majG6}
\end{align*}

\textcolor{black}{We now have all the inequalities needed to control $\|F(v+w)\|_{\mathcal{C}^{-1-\varepsilon}}$ and $\|G(v,w)\|_{\mathcal{C}^{-\frac{1}{2}-\varepsilon}}$ with $\|(v,w)\|_{\Xi}$. In the following Lemma \textcolor{black}{we show} that if we assume a control of order $h$ on $\|(v,w)\|_{\Xi}$ and on the norms of the symbols of $\mathcal{T}$, it \textcolor{black}{implies} a control of order $h^2$ on $F$ and $G$. Such a control will reveal extremely powerful for $h\ll 1$.} 
\begin{lemma}\label{majFG}
    Let us take $h\in(0,1)$ and a given time $t_0\in\textcolor{black}{[0,T]}$. \textcolor{black}{We make three assumptions. First that}
    \begin{align*}&\textcolor{black}{\|(v,w)\|_{\Xi,t_0}} \\ &=\max(\sup_{s\in [0,t_0]} \|v(s)\|_{\mathcal{C}^{1-2\varepsilon}},\sup_{s\in[0,t_0]}\|w(s)\|_{\mathcal{C}^{\frac{3}{2}-2\varepsilon}}, [v_{|[0,t_0]}]_{0,\frac{1}{8}},[w_{|[0,t_0]}]_{0,\frac{1}{8}})\le h, \numberthis \label{hyp1}\end{align*}
    second that for all symbols $\tau\in \mathcal{T}\setminus\{\RSV\}$ and $k\le n_{\tau}$, \textcolor{black}{denoting $\alpha_{\tau}=|\tau|-\frac{\varepsilon}{3}$}
    \begin{equation}\text{}\sup_{s\in[0,t_0]}\|\Pi_k\tau(s)\|_{\mathcal{C}^{\alpha_{\tau}}}\le h^{k},\label{hyp2}\end{equation}
    \textcolor{black}{and third that}
\begin{equation}\textcolor{black}{(1\wedge T^{\frac{\lambda}{4}})^{-1}\sup_{s\in[0,t_0]}\|\RSV(s)\|_{\mathcal{C}^{1-\frac{\varepsilon}{3}}}\le h^2, \ \ \ \ \ \ (1\wedge T^{\frac{\lambda}{2}})[\RSIW_{|[0,t_0]}]_{0,\frac{1}{8}}\le h^3}.\label{hyp3}\end{equation}
\textcolor{black}{Then we have that}
\begin{equation}\sup_{s\in [0,t_0]} \|F(v+w)(s)\|_{\mathcal{C}^{-1-\varepsilon}}\le C h^2 \ \ \ \ \ \ \text{and }  \ \ \ \ \sup_{s\in [0,t_0]} \|G(v+w)(s)\|_{\mathcal{C}^{-\frac{1}{2}-\varepsilon}}\le C h^2 \label{majp}\end{equation}
\end{lemma}

\textcolor{black}{The proof is left to Section \ref{auxl}.}

\begin{remark}\label{ts}
    The second bound of Lemma \ref{majFG} would not hold if we \textcolor{black}{had not} removed $3I(\RSWW)$ from the definition of $w$. Indeed, the term $d_0$ would therefore contain a $3\RSWW$ whose component in the first Wiener chaos is only bounded by an $h$ (and not $h^2$) in our hypothesis.
\end{remark}

\textcolor{black}{Under \eqref{hyp1},\eqref{hyp2} and \eqref{hyp3}, we can combine Lemma \ref{bornefg} and Lemma \ref{majFG} to get that}
\begin{align*}
    &\textcolor{black}{\|(v,w)\|_{\Xi,t_0}}\\ &=\max(\sup_{s\in [0,t_0]} \|v(s)\|_{\mathcal{C}^{1-2\varepsilon}},\sup_{s\in[0,t_0]}\|w(s)\|_{\mathcal{C}^{\frac{3}{2}-2\varepsilon}},[v_{|[0,t_0]}]_{0,\frac{1}{8}},[w_{|[0,t_0]}]_{0,\frac{1}{8}})\\ & \le Ch^2. \numberthis \label{majvw}
\end{align*}
\textcolor{black}{With that in mind,} we can \textcolor{black}{now} end the proof of Theorem \ref{thp}. Let us take $h_0\in (0,1)$ such that $\textcolor{black}{C}h_0^2<h_0$. Then, considering for $h\in (0,h_0)$
 \begin{align*}&\kappa=\inf\left\{t>0, \, \textcolor{black}{\|(v,w)\|_{\Xi,t}}> h \right\}\wedge T\end{align*}
we get
\begin{align*}&\{\kappa<T\}\cap \{\forall \tau\in \mathcal{T}, \, \forall k\le n_{\tau}, \, \sup_{s\le T} \|\Pi_k\tau(s)\|_{\mathcal{C}^{\alpha_{\tau}}}\le h^{k}\} \\& \ \ \ \ \ \ \ \ \cap \{[\RSIW]_{0,\frac{1}{8}}\le \textcolor{black}{h^3(1\wedge T^{\frac{\lambda}{2}})^{-1}}\}\textcolor{black}{\cap \{\sup_{s\le T}\|\RSV(s)\|_{\mathcal{C}^{-1-\frac{\varepsilon}{3}}}\le h^2(1\wedge T^{\frac{\lambda}{4}})\}}=\varnothing.\end{align*}
\textcolor{black}{Indeed, let us assume by contradiction that there is an $\omega$ in this intersection. Then, we have for this $\omega$ that $\kappa<T$, and since $t\mapsto \|(v,w)\|_{\Xi,t}$ is continuous (see Lemma \ref{cont}) and $\|(v,w)\|_{\Xi,0}=0$, it implies that $\|(v,w)\|_{\Xi,\kappa}=h$. We can assume without loss of generality that $[v_{|[0,\kappa]}]_{0,\frac{1}{8}}=h$, but since \eqref{hyp1}, \eqref{hyp2} and \eqref{hyp3} hold for $t_0=\kappa(\omega)$, we can invoke \eqref{majvw} to get that that $h=[v_{|[0,\kappa]}]_{0,\frac{1}{8}}\le Ch^2<h$ and we have a contradiction.}\\

Therefore, using Theorem \ref{cs} and Corollary \ref{chc}, we get that
\begin{align*}
    &\mathbb{P}(\kappa<T) \\ &\le \mathbb{P}(\exists (\tau,k)\in (\mathcal{T},\llbracket 1,n_{\tau} \rrbracket), \, \sup_{s\in [0,T]} \|\Pi_k\tau(s)\|_{\mathcal{C}^{\alpha_{\tau}}}>h^{k}) \\
    & \ \ \ \ \ \ + \mathbb{P}([\RSIW]_{0,\frac{1}{8}}>\textcolor{black}{h^3(1\wedge T^{\frac{\lambda}{2}})^{-1}})\textcolor{black}{+\mathbb{P}(\sup_{s\in[0,T]}\|\RSV(s)\|_{\mathcal{C}^{-1-\frac{\varepsilon}{3}}}>h^2(1\wedge T^{\frac{\lambda}{4}}))} \\
    &\le\sum_{\tau\in \mathcal{T}}\sum_{k=1}^{n_{\tau}}\mathbb{P}(\sup_{s\in [0,T]} \|\Pi_k \tau(s)\|_{\mathcal{C}^{\alpha_{\tau}}}>h^{k}) \\
    & \ \ \ \ \ \ + \mathbb{P}([\RSIW]_{0,\frac{1}{8}}>\textcolor{black}{h^3(1\wedge T^{\frac{\lambda}{6}}})^{-3})\textcolor{black}{+\mathbb{P}(\sup_{s\in[0,T]}\|\RSV(s)\|_{\mathcal{C}^{-1-\frac{\varepsilon}{3}}}>h^2(1\wedge T^{\frac{\lambda}{4}}))} \\
    &\le 5\text{Card}(\mathcal{T}) D'\exp(-\frac{C'}{\max(T^{\frac{\lambda}{5}},T^{\lambda})}\frac{h^2}{\sigma^2})+m_{\RSIW}'\exp(-\frac{\ell_{\RSIW}'}{\textcolor{black}{1\wedge T^{\frac{\lambda}{3}}}} \frac{h^2}{\sigma^2}) \\
    & \ \ \ \ \ \ \textcolor{black}{+d_{\RSV}\exp(-\frac{c_{\RSV}}{T^{\frac{\lambda}{2}}}(1\wedge T^{\frac{\lambda}{4}})\frac{h^2}{\sigma^2})} \\
    &\le 5\text{Card}(\mathcal{T}) D'\exp(-\frac{C'}{\max(T^{\frac{\lambda}{5}},T^{\lambda})}\frac{h^2}{\sigma^2})+m_{\RSIW}'\exp(-\frac{\ell_{\RSIW}'}{\textcolor{black}{T^{\frac{\lambda}{3}}}} \frac{h^2}{\sigma^2}) \\
    & \ \ \ \ \ \ \textcolor{black}{+d_{\RSV}\exp(-\frac{c_{\RSV}}{\max(T^{\frac{\lambda}{4}},T^{\frac{\lambda}{2}})}\frac{h^2}{\sigma^2})}
\end{align*}
since $T^{\frac{\lambda}{k}}\le \max(T^{\frac{\lambda}{5}},T^{\lambda})$. Here we may take $D'=\sup_{\tau\in \mathcal{T}} d_{\tau}(\lambda,\textcolor{black}{a_{+},a_{-}})$ et $C'=\inf_{\tau\in \mathcal{T}} c_{\tau}(\lambda,\textcolor{black}{a_{+},a_{-}})$. We then take $D=5\text{Card}(\mathcal{T}) D'+m_{\RSIW}'\textcolor{black}{+d_{\RSV}}$ and $C=\min(C',\ell_{\RSIW}',\textcolor{black}{c_{\RSV}})$ and Theorem \ref{thp} is proven.

\subsubsection{\textcolor{black}{Proof of Corollary \ref{corp}}}
\textcolor{black}{We want to control $\sup_{t\in [0,T]}\|(\varphi-\overline{\phi})(t)\|_{\mathcal{C}^{-\frac{1}{2}-\varepsilon}}$. We recall that}
$$\textcolor{black}{\varphi(t)-\overline{\phi}(t)=\RSI(t)-\RSIW(t)+3I(\RSWW)(t)+(v+w)(t)}$$
\textcolor{black}{We now consider $h\in (0,h_0)$, where $h_0$ is given by Theorem \ref{thp}. We have then}
\begin{align*}
    &\mathbb{P}(\sup_{t\in [0,T]}\|(\varphi-\overline{\phi})(t)\|_{\mathcal{C}^{-\frac{1}{2}-\varepsilon}}>h) \\ &\le \mathbb{P}(\sup_{t\in[0,T]} \|\RSI(t)\|_{\mathcal{C}^{-\frac{1}{2}-\varepsilon}}>\frac{h}{7})+\mathbb{P}(\sup_{t\in[0,T]} \|\RSIW(t)\|_{\mathcal{C}^{-\frac{1}{2}-\varepsilon}}>\frac{h}{7}) \\ & \ \ \ \ +\mathbb{P}(\sup_{t\in[0,T]} 3\|I(\RSWW(t))\|_{\mathcal{C}^{-\frac{1}{2}-\varepsilon}}>\frac{h}{7})\\ 
    & \ \ \ \  \ \ \ \ +\mathbb{P}(\sup_{t\in[0,T]} \|v(t)\|_{\mathcal{C}^{-\frac{1}{2}-\varepsilon}}>\frac{h}{7})+\mathbb{P}(\sup_{t\in[0,T]} \|w(t)\|_{\mathcal{C}^{-\frac{1}{2}-\varepsilon}}>\frac{h}{7})
\end{align*}
\textcolor{black}{We want to use Theorem \ref{cs} and Theorem \ref{thp} to deal with all of those terms, but we notice that $I(\RSWW)$ is not in $\mathcal{T}$ and we cannot therefore use directly Theorem \ref{cs} on it. However, thanks to Schauder's estimate we have}
$\|I(\RSWW)(t)\|_{\mathcal{C}^{-\frac{1}{2}-\varepsilon}}\le \|I(\RSWW)(t)\|_{\mathcal{C}^{\frac{3}{2}-\textcolor{black}{2}\varepsilon}}\le K \sup_{s\in [0,t]} \|\RSWW(s)\|_{\mathcal{C}^{-\frac{1}{2}-\varepsilon}},$
\textcolor{black}{with $K$ independent of time. We now use that $h^5\le h^3\le h$ and the chaos decomposition $\RSWW(t)\in\mathcal{H}_1\bigoplus\mathcal{H}_3\bigoplus \mathcal{H}_5$}
\begin{align*}
    &\mathbb{P}(\sup_{t\in[0,T]} \|\RSWW(t)\|_{\mathcal{C}^{-\frac{1}{2}-\varepsilon}}>\frac{h}{7}) \\
    &\le \mathbb{P}(\sup_{t\in[0,T]} \|\Pi_1\RSWW(t)\|_{\mathcal{C}^{-\frac{1}{2}-\varepsilon}}>\frac{h}{21})+\mathbb{P}(\sup_{t\in[0,T]} \|\Pi_3\RSWW(t)\|_{\mathcal{C}^{-\frac{1}{2}-\varepsilon}}>\frac{h}{21}) \\ & \ \ \ \ \ \ +\mathbb{P}(\sup_{t\in[0,T]} \|\Pi_5\RSWW(t)\|_{\mathcal{C}^{-\frac{1}{2}-\varepsilon}}>\frac{h}{21}) \\
    &\le \mathbb{P}(\sup_{t\in[0,T]} \|\Pi_1\RSWW(t)\|_{\mathcal{C}^{-\frac{1}{2}-\varepsilon}}>\frac{h}{21})+\mathbb{P}(\sup_{t\in[0,T]} \|\Pi_3\RSWW(t)\|_{\mathcal{C}^{-\frac{1}{2}-\varepsilon}}>\frac{h^3}{21}) \\ & \ \ \ \ \ \ +\mathbb{P}(\sup_{t\in[0,T]} \|\Pi_5\RSWW(t)\|_{\mathcal{C}^{-\frac{1}{2}-\varepsilon}}>\frac{h^5}{21}) \\
    &\le 3d_{\RSWW}\exp(-\frac{c_{\RSWW}}{441\max(T^\frac{\lambda}{5},T^{\lambda})}\frac{h^2}{\sigma^2})
\end{align*}
\textcolor{black}{Finally, we have}
\begin{align*}
    &\mathbb{P}(\sup_{t\in [0,T]}\|(\varphi-\overline{\phi})(t)\|_{\mathcal{C}^{-\frac{1}{2}-\varepsilon}}>h) \\ &\le \mathbb{P}(\sup_{t\in[0,T]} \|\RSI(t)\|_{\mathcal{C}^{-\frac{1}{2}-\varepsilon}}>\frac{h}{7})+\mathbb{P}(\sup_{t\in[0,T]} \|\RSIW(t)\|_{\mathcal{C}^{\frac{1}{2}-\varepsilon}}>\frac{h^3}{7}) \\ 
    & \ \ \ \ +\mathbb{P}(\sup_{t\in[0,T]} \|\RSWW(t)\|_{\mathcal{C}^{-\frac{1}{2}-\varepsilon}}>\frac{h}{7K})\\
    & \ \ \ \ \ \ \ \  +\mathbb{P}(\sup_{t\in[0,T]} \|v(t)\|_{\mathcal{C}^{1-2\varepsilon}}>\frac{h}{7})+\mathbb{P}(\sup_{t\in[0,T]} \|w(t)\|_{\mathcal{C}^{\frac{3}{2}-2\varepsilon}}>\frac{h}{7}) \\
    &\le d_{\RSI}\exp(-\frac{c_{\RSI}}{49 T^{\lambda}}\frac{h^2}{\sigma^2})+d_{\RSIW}\exp(-\frac{c_{\RSIW}}{7^{\frac{2}{3}} T^{\frac{\lambda}{3}}}\frac{h^2}{\sigma^2}) \\
    & \ \ \ \ \ +3d_{\RSWW}\exp(-\frac{c_{\RSWW}}{441\max(T^\frac{\lambda}{5},T^{\lambda})K^2}\frac{h^2}{\sigma^2}) \\
    & \ \ \ \ \ \ \ \ \ \ +2D\exp(-\frac{C}{49\max(T^{\frac{\lambda}{5}},T^{\lambda})}\frac{h^2}{\sigma^2}) \\
    &\le D'\exp(-\frac{C'}{\max(T^{\frac{\lambda}{5}},T^{\lambda})}\frac{h^2}{\sigma^2})
\end{align*}
\textcolor{black}{where $C'$ and $D'$ are independent of $T$. We therefore have the result.}

\subsection{\label{auxl}\textcolor{black}{Proof of auxiliary Lemmas}}

\subsubsection{\textcolor{black}{Proof of Lemma \ref{bornefg}}}
    \textcolor{black}{In this proof we will use massively the fact that}
$$v(t)=\int_0^t e^{\alpha(t,r)}e^{(t-r)\Delta}F(v+w)(r)\mathrm{d}r,$$
\textcolor{black}{for $\alpha(t,r)=\int_r^t a(s)\mathrm{d}s$}.\\

    \textcolor{black}{To prove the bound on $\|v\|$,} \textcolor{black}{we use} Schauder's estimate (Proposition \ref{es}), the inequality $\sup a=-a_{+}<0$ and the fact that $\alpha\mapsto \|\cdot\|_{\mathcal{C}^{\alpha}}$ is non-decreasing:
\begin{align*}
    \|v(t)\|_{\mathcal{C}^{1-2\varepsilon}}&=\|\int_0^t e^{\alpha(t,r)}e^{(t-r)\Delta}F(v+w)(r)\mathrm{d}\textcolor{black}{r}\|_{\mathcal{C}^{1-2\varepsilon}} \\
    &\le\int_0^t e^{\alpha(t,r)}\|e^{(t-r)\Delta}F(v+w)(r)\|_{\mathcal{C}^{1-2\varepsilon}} \mathrm{d}r \\
    &\le\int_0^t e^{-a_{+}(t-r)} C \, (t-r)^{\frac{-1-\varepsilon-(1-2\varepsilon)}{2}}\|F(v+w)(r)\|_{\mathcal{C}^{-1-\varepsilon}} \mathrm{d}r \\
    &\le C\sup_{r\in [0,t]}\|F(v+w)(r)\|_{\mathcal{C}^{-1-\varepsilon}} \int_0^t e^{-a_{+}r} r^{-1+\frac{\varepsilon}{2}}\mathrm{d}r \\
    &\le C \sup_{r\in [0,t]}\|F(v+w)(r)\|_{\mathcal{C}^{-1-\varepsilon}} .
\end{align*}
\textcolor{black}{where $C$ is independent of $T$ since $r\mapsto e^{-a_{+}r} r^{-1+\frac{\varepsilon}{2}}$ is integrable. This inequality holds for all $t\in [0,T]$, so we can take the supremum in $t\in[0,t_0]$ on both sides to get \eqref{majv}. The same method applied on $w$ gives us \eqref{majw}.}\\

\textcolor{black}{To prove the bound on $[v]$}, we will use a trick found in \cite{MW}. We have for $0\le s<t\le T$ that
\begin{align*}
    v(t)-v(s)&=\int_s^t e^{\alpha(t,r)} e^{(t-r)\Delta}F(v+w)(r) \mathrm{d}r \\
    & \ \ \ \ \ +(e^{\alpha(t,s)}-1)\int_0^s e^{\alpha(s,r)} e^{(t-r)\Delta}F(v+w)(r) \mathrm{d}r \\
    & \ \ \ \ \ \ \ \ \ \ +(e^{(t-s)\Delta}-1)\int_0^s e^{\alpha(s,r)} e^{(s-r)\Delta}F(v+w)(r) \mathrm{d}r.   
\end{align*}
We then bound the $L^{\infty}$ norm of each of these three terms. For the first one we have
\begin{align*}
    &\|\int_s^t e^{\alpha(t,r)} e^{(t-r)\Delta}F(v+w)(r) \mathrm{d}r\|_{L^{\infty}} \\ &\le \int_s^t e^{\alpha(t,r)} \|e^{(t-r)\Delta}F(v+w)(r)\|_{L^{\infty}}\mathrm{d}r \\
    &\le C\int_s^t e^{-a_+(t-r)} (t-r)^{-\frac{1+\varepsilon}{2}} \|F(v+w)(r)\|_{\mathcal{C}^{-1-\varepsilon}}\mathrm{d}r \\
    &\le C \sup_{r\in [0,t]} \|F(v+w)(r)\|_{\mathcal{C}^{-1-\varepsilon}} \int_s^t e^{-a_+(t-r)}  (t-r)^{-\frac{1+\varepsilon}{2}} \mathrm{d}r  \\
    &\le C \sup_{r\in [0,t]} \|F(v+w)(r)\|_{\mathcal{C}^{-1-\varepsilon}} 1\wedge(t-s)^{1-\frac{1+\varepsilon}{2}}\\
    &\le C \sup_{r\in [0,t]} \|F(v+w)(r)\|_{\mathcal{C}^{-1-\varepsilon}} (t-s)^{\frac{1}{8}}.
\end{align*}
where $C$ is independent of $T$. We then consider the second term \textcolor{black}{and use} the inequalities \textcolor{black}{\eqref{inc}}
\begin{align*}
   & \|(e^{\alpha(t,s)}-1)\int_0^s e^{\alpha(s,r)} e^{(t-r)\Delta}F(v+w)(r) \mathrm{d}r\|_{L^{\infty}} \\ &\le |e^{\alpha(t,s)}-1| \int_0^s e^{-a_{+}(s-r)} \|e^{(t-u)\Delta}F(v+w)(r)\|_{L^{\infty}} \\
    &\le C \sup_{r\in [0,t]} \|F(v+w)(r)\|_{\mathcal{C}^{-1-\varepsilon}}(1\wedge a_{-} (t-s))\int_0^s e^{-a_{+}(t-r)} (t-r)^{-\frac{1+\varepsilon}{2}}\mathrm{d}r \\
    &\le C \sup_{r\in [0,t]} \|F(v+w)(r)\|_{\mathcal{C}^{-1-\varepsilon}}(1\wedge a_{-} (t-s)) \\
    &\le C \sup_{r\in [0,t]} \|F(v+w)(r)\|_{\mathcal{C}^{-1-\varepsilon}} (t-s)^{\frac{1}{8}}
\end{align*}
Finally, for the third term, we use the second inequality of Proposition \ref{es} to control the operator $(e^{(t-s)\Delta}-1)$
\begin{align*}
    &\|(e^{(t-s)\Delta}-1)\int_0^s e^{\alpha(s,r)} e^{(s-r)\Delta}F(v+w)(r) \mathrm{d}r \|_{L^{\infty}} \\ &\le (t-s)^{\frac{\frac{1}{4}-0}{2}}\|\int_0^s e^{\alpha(s,r)} e^{(s-r)\Delta}F(v+w)(r) \mathrm{d}r\|_{\mathcal{C}^{\frac{1}{4}}} \\
    &\le C (t-s)^{\frac{1}{8}}\int_0^s e^{-a_{+}(s-r)}(s-r)^{-\frac{1+\varepsilon+\frac{1}{4}}{2}}\|F(v+w)(r)\|_{\mathcal{C}^{-1-\varepsilon}}\mathrm{d}r \\
    &\le C (t-s)^{\frac{1}{8}} \sup_{r\in [0,t]} \|F(v+w)(r)\|_{\mathcal{C}^{-1-\varepsilon}} \int_0^s e^{-a_{+}(s-r)}(s-r)^{-\frac{1+\varepsilon+\frac{1}{4}}{2}} \\
    &\le C (t-s)^{\frac{1}{8}} \sup_{r\in [0,t]} \|F(v+w)(r)\|_{\mathcal{C}^{-1-\varepsilon}}
\end{align*}
Dividing by $(t-s)^{\frac{1}{8}}$ on both sides of the three inequalities and taking the supremum in $0\le s<t\le t_0$, we eventually get \textcolor{black}{\eqref{lipv}}. Following exactly the same strategy \textcolor{black}{for $w$}, we get \textcolor{black}{\eqref{lipw}}.

\subsubsection{\textcolor{black}{Proof of Lemma \ref{majFG}}}

 \textcolor{black}{The key idea is} is that under \textcolor{black}{hypothesis} \eqref{hyp1},\eqref{hyp2} \textcolor{black}{and \eqref{hyp3}} \textcolor{black}{we can go back to \eqref{majF} for $F$ and \eqref{majG1},\eqref{majG2},\eqref{majG3},\eqref{majG4},\eqref{majG5} and \eqref{majG6} for $G$, and replace all occurrences of $\|v\|_{\mathcal{C}^{1-2\varepsilon}}$, $\|w\|_{\mathcal{C}^{\frac{3}{2}-2\varepsilon}}$, $[v_{|[0,t_0]}]_{0,\frac{1}{8}}$ and $[w_{|[0,t_0]}]_{0,\frac{1}{8}}$ by $h$, the one occurrence of $[\RSIW_{|[0,t_0]}]_{0,\frac{1}{8}}$ by $h^3(1\wedge T^{\frac{\lambda}{2}})^{-1}$ and all occurrences of $\|\tau\|$ by $h^{k_{\tau}}$ for $k_{\tau}$ the smallest $k$ such that $\tau$ has a non-zero component in the Wiener chaos of order $k$, \textcolor{black}{except from $\|\RSV\|$ which is replaced by $h^2(1\wedge T^{\frac{\lambda}{4}})$}.} \\

We \textcolor{black}{then} observe that the powers of $h$ \textcolor{black}{in} \eqref{majF},\eqref{majG1},\eqref{majG2},\eqref{majG3},\eqref{majG4},\eqref{majG5} and \eqref{majG6} are always superior or equal to $2$. Since $h\in(0,1)$, we have that $h^n\le h^2$ if $n\ge 2$, and therefore we can bound $F$ and $G$ by a constant times $h^2$, which is the wanted result. \\
\newpage

\setcounter{section}{0}
\renewcommand\thesection{\Alph{section}}
\section{Technical tools}
\textcolor{black}{The content of this section can essentially be found in \cite{MW} and \cite{MWX}, with some improvements derived from \cite{H}.}
\subsection{Decomposition into Paley-Littlewood blocks \label{PL}}
There exists $\tilde{\chi},\chi\in \mathcal{C}_c^{\infty}(\R^d)$ taking values in $[0,1]$ with
$$\text{Supp} \ \tilde{\chi} \subset B(0,\frac{4}{3}), \ \ \ \ \ \ \ \ \text{Supp} \ \chi \subset B(0,\frac{8}{3})\setminus B(0,\frac{3}{4})$$
and such that
$$\tilde{\chi}(\zeta)+\sum_{k=0}^{\infty} \chi\left(\frac{\zeta}{2^k}\right)=1, \ \ \ \ \ \forall \zeta \in \R^d.$$
We furthermore assume that $\tilde{\chi}$ and $\chi$ are radially symmetric. We write
$$\chi_{-1}=\tilde{\chi}, \ \ \ \ \ \chi_k(\cdot):=\chi\left(\frac{\cdot}{2^k}\right) \ \ \ k\ge 0.$$
These objects allow us to define $\|\cdot\|_{\mathcal{C}^{\alpha}}$ for all $\alpha\in \R$ in a way that is consistent with the common definition for $\alpha\in (0,1)$. Indeed, \textcolor{black}{denoting $\mathcal{F}$ the Fourier transform, we consider} for $f\in \mathcal{C}^{\infty}(\mathbb{T}^d)$ and $k\ge -1$
$$\delta_k f :=\mathcal{F}^{-1}(\chi_k \hat{f}),$$
\textcolor{black}{and} we define the norm of $\|\cdot\|_{\mathcal{C}^{\alpha}}$ by
$$\|f\|_{\mathcal{C}^{\alpha}}:=\sup_{k\ge -1} 2^{\alpha k} \|\delta_k f\|_{L^{\infty}}.$$
This quantity is finite for all $f\in \mathcal{C}^{\infty}(\mathbb{T}^d)$, and the space $\mathcal{C}^{\alpha}(\mathbb{T}^d)$ is therefore defined as the completion of $\mathcal{C}^{\infty}$ for this norm. 
\begin{remark}
    We do not have necessarily that $\|f\|_{\mathcal{C}^{\alpha}}<+\infty$ implies $f\in \mathcal{C}^{\alpha}$ but we can verify that if a distribution $f$ satisfies $\|f\|_{\mathcal{C}^{\alpha}}<+\infty$, then $f\in \mathcal{C}^{\beta}$ for all $\beta<\alpha$.
\end{remark}
We now state one of the most crucial Lemmas of the theory of Besov spaces. It corresponds to \cite[Lemma 2.\textcolor{black}{1}]{H}, we will here only state the first part.
\begin{lemma}(Bernstein's Lemma)
    Let $B$ be the unit ball. There exists a constant $C$ such that for any $n\ge 0$, any couple $(p,q)\in [1,+\infty]^2$ with $q\ge p\ge 1$ and any function $u$ of $L^p$ we have
    $$\text{Supp} \ \hat{u}\subset \lambda B \  \Rightarrow \ \sup_{|\alpha|=n}\|\partial^{\alpha} u\|_{L^q}\le C^{n+1}\lambda^{n+d(\frac{1}{p}-\frac{1}{q})}\|u\|_{L^p}.$$
\end{lemma}
\begin{remark}
    The theorem is generally written for $L^p=L^p(\R^d)$, but since it is a direct consequence of Young's inequality, and since the Fourier transform maps functions on $\mathbb{T}^d$ to functions on $\Z^d\subset \R^d$, it also works for $L^p=L^p(\mathbb{T}^d)$.
\end{remark}
In the setting of Paley-Littlewood theory, we take $B=B(0,\frac{8}{3})$, $\lambda=2^{\textcolor{black}{k}}$ and $q=+\infty$ and we get for $n=0$ that
\begin{equation}\|\delta_k u\|_{L^{\infty}}\le C 2^{\frac{dk}{p}}\|\delta_k u\|_{L^p}\label{bernstein}\end{equation}
for all $k\ge -1$, $p\ge 1$ and $u\in L^1(\mathbb{T}^d)$. This inequality is key to prove the following boundedness criterion:
\begin{proposition}\label{critere}
    Let $\beta<\alpha$. There exists $C_0$ such that for $p>\frac{d}{\alpha-\beta}+1$ we have for every random distribution $f$ on $\mathbb{T}^d$ that
    $$\mathbb{E}(\|f\|_{\mathcal{C}^{\beta}}^p)\le C_0^p \sup_{k\ge -1} 2^{\alpha k p}\mathbb{E}[\|\delta_k f\|_{L^p}^p]$$
\end{proposition}
\begin{proof}
    By definition of the $\mathcal{C}^{\beta}$ norm and then using \eqref{bernstein} we have 
    $$\|f\|_{\mathcal{C}^\beta}^p=\sup_{k\ge -1} 2^{\beta k p} \|\delta_k\|_{L^{\infty}}^p\le C^p \sup_{k\ge -1} 2^{ k (\beta p+d)}\|\delta_k f\|_{L^p}^p$$
    To take the expectation of $\|\delta_k f\|_{L^p}^p$ directly, we enlarge the supremum on the right-hand side to a sum and we get 
    $$\mathbb{E} \|f\|_{\mathcal{C}^{\beta}}^p \le C^p \sum_{k\ge -1} 2^{k(\beta p+d)}\mathbb{E} \|\delta_k f\|_{L^p}^p=C^p \sum_{k\ge -1} 2^{kp(\beta+\frac{d}{p}-\alpha)}2^{\alpha kp}\mathbb{E} \|\delta_k f\|_{L^p}^p.$$
    The sum $\sum_{k\ge -1} 2^{kp(\beta+\frac{d}{p}-\alpha)}$ can be bounded uniformly in $p>\frac{d}{\alpha-\beta}+1$. So taking $C_0$ slightly larger than $C$ we have the result.
\end{proof}
We now state Schauder's estimates which are extremely useful when we are dealing with the heat semigroup.
\begin{proposition}[Proposition A.13 of \cite{MW}]\label{es}
Let $\alpha,\beta\in \R$:
\begin{itemize}
    \item If $\alpha\ge \beta$ there exists $C>0$ such that, uniformly in $t>0$ and $f\in \mathcal{C}^{\beta}$, we have
    $$\|e^{t \Delta} f\|_{\mathcal{C}^{\alpha}}\le C t^{\frac{\beta-\alpha}{2}}\|f\|_{\mathcal{C}^{\beta}}$$
    \item If $0\le\beta-\alpha\le 2$ there exists $C>0$ such that uniformly in $t\ge 0$ and $f\in \mathcal{C}^{\beta}$ we have
    $$\|(1-e^{t\Delta})f\|_{\mathcal{C}^{\alpha}}\le C t^{\frac{\beta-\alpha}{2}} \|f\|_{\mathcal{C}^{\beta}}$$
\end{itemize}

\end{proposition}

\subsection{Paracontrolled calculus \label{PC}}
As we said above, the main tool of paracontrolled calculus is the decomposition of the standard product of functions into three different objects we will now describe explicitly:
$$fg=\sum_{k<l-1} \delta_k f \delta_l g+ \sum_{|k-l|\le 1} \delta_k f \delta_l g+\sum_{k>l+1} \delta_k f \delta_l g.$$
We write the first term $f\pl g$, the second one $f \pe g$ and the third one $f \pg g$. The paraproducts $f\pl g$ and $f \pg g=g\pl f$ are defined for all functions $f$ and $g$ with at least (negative) Hölder regularity. The resonant product $f\pe g$ is only defined for functions with "compensating" Hölder regularity, and is the reason why the standard product is often not defined. More precisely, both for the paraproducts and the resonant product we have the following mappings:
\begin{proposition}[Proposition A.7 of \cite{MW}]\label{ip}
    Let $\alpha,\beta\in \R$:
    \begin{itemize}
        \item If $\alpha+\beta>0$, then the mapping $(f,g)\mapsto f\pe g$ extends to a continuous bilinear map from $\mathcal{C}^{\alpha}\times \mathcal{C}^{\beta}$ to $\mathcal{C}^{\alpha+\beta}$.
        \item The mapping $(f,g)\mapsto f\pl g$ extends to a continuous bilinear map from $L^{\infty}\times \mathcal{C}^{\beta}$ to $\mathcal{C}^{\beta}$.
        \item If $\alpha<0$, then the mapping $(f,g)\mapsto f\pl g$ extends to a continuous bilinear map from $\mathcal{C}^{\alpha}\times \mathcal{C}^{\beta}$ to $\mathcal{C}^{\alpha+\beta}$. 
        \item If $\alpha<0<\beta$ and $\alpha+\beta>0$ then the mapping $(f,g)\mapsto fg$ extends to a bilinear map from $\mathcal{C}^{\alpha}\times \mathcal{C}^{\beta}\to \mathcal{C}^{\alpha}$.
    \end{itemize}
\end{proposition}
Now we will state two technical Lemmas that illustrate that the resonant product, the paraproduct\textcolor{black}{s} and the heat semigroup have powerful interactions with one another. We start with the commutator of $\pl$ and $\pe$.
\begin{proposition}[Proposition A.9 of \cite{MW}]\label{com}
    Let $\alpha<1$ and $\beta,\gamma\in \R$ such that $\beta+\gamma<0$ and $\alpha+\beta+\gamma>0$. Then the mapping
    $$[\pl,\pe]:(f,g,h)\longmapsto (f\pl g)\pe h-f(g\pe h)$$
    extends to a continuous trilinear map from $\mathcal{C}^{\alpha}\times \mathcal{C}^{\beta}\times \mathcal{C}^{\gamma}$ \textcolor{black}{to} $\mathcal{C}^{\alpha+\beta+\gamma}$.
\end{proposition}
\begin{remark} As said in Section \ref{resultats}, this result is quite remarkable because the resonant product $g\pe h$ is not supposed to be defined for $\beta+\gamma<0$.\end{remark}

Finally we consider the commutator of $e^{t\Delta}$ and $\pl$.
\begin{proposition}[Proposition A.16 of \cite{MW}]\label{come}
Let $\alpha<1, \, \beta\in \R, \, \gamma\ge \alpha+\beta$. For every $t\ge 0$ define
$$[e^{t\Delta},\pl]:(f,g)\mapsto e^{t\Delta}(f\pl g)-f \pl(e^{t\Delta}g).$$
There exists $C<+\infty$ such that, uniformly over $t>0$,
$$\|[e^{t\Delta},\pl](f,g)\|_{\mathcal{C}^{\gamma}}\le C t^{\frac{\alpha+\beta-\gamma}{2}}\|f\|_{\mathcal{C}^{\alpha}}\|g\|_{\mathcal{C}^{\beta}}.$$
\end{proposition}

\subsection{Wiener chaos and Nelson estimate}
Let us consider $\xi$ the space-time white noise, $\xi$ is a distribution and therefore $\xi(t,x)$ is not well-defined. We can however define for all $\varphi\in L^2(\R\times \mathbb{T}^d)$ the quantity $\int_{\R\times\mathbb{T}^d}\varphi(z)\xi(\mathrm{d}z)=\xi(\varphi)$ that verifies
$$\mathbb{E}[\xi(\varphi)^2]= \|\varphi\|_{L^2(\R\times \mathbb{T}^d)}^2.$$
We can furthermore define iterated Wiener-Itô integrals based on $\xi$\textcolor{black}{,} written $\xi^{\otimes k}(\varphi)$ for $k\ge 1$ and $\varphi\in L^2((\R\times \mathbb{T}^d)^k)$. We usually write
$$\xi^{\otimes k}(\varphi)=\int_{(\R\times \mathbb{T}^d)^k}\varphi(z_1,\dots,z_k)\xi(\mathrm{d}z_1)\cdots\xi(\mathrm{d}z_k).$$
We then consider
$$\mathcal{H}_k:=\{\xi^{\otimes}(\varphi),\ \varphi\in L^2((\R\times\mathbb{T}^d)^k)\}$$
which is the $k$-th homogeneous Wiener chaos with $\mathcal{H}_0=\mathbb{R}$. Since we have
$$L^2(\Omega,\mathcal{A},\mathbb{P})=\bigoplus_{k=0}^{+\infty} \mathcal{H}_k$$
we can for all $k\ge 0$ consider $\Pi_k$ the projection on $\mathcal{H}_k$. We have the following property:
\begin{lemma}
    For each $n\in \N$, the closure in $L^2(\Omega,\mathcal{A},\mathbb{P})$ of the linear span of the set
    $$\{\xi(\varphi_1)\cdots \xi(\varphi_k), \ k\le n , \ \varphi_1,\dots,\varphi_k\in L^2(\R\times \mathbb{T}^d)\}$$
    coincides with
    $$\mathcal{H}_{\le n}:=\bigoplus_{k=0}^n \mathcal{H}_k .$$
\end{lemma}
Let us now state Nelson's estimate which is a key inequality when one deals with Wiener chaos.
\begin{proposition} \label{nelson}
    For every $n\ge 1$, there exists a constant $C_n<+\infty$ such that for every $X\in \mathcal{H}_{\le n}$ and $p\ge 2$ we have
    $$\mathbb{E}[|X|^p]^{\frac{1}{p}}\le C_n (p-1)^\frac{n}{2}\mathbb{E}(X^2)^{\frac{1}{2}}$$
\end{proposition}

\section*{Acknowledgments}
I would like to thank my PhD supervisor Nils Berglund, who introduced me to the subject, provided me with the necessary bibliography to work on it, and edited my article several times. \textcolor{black}{I would also like to thank the referee for their careful reading of my article and the very detailed report they sent me.} This work was supported by my PhD grant from the \'Ecole Normale Sup\'erieure. 

\end{document}